\documentclass[10pt]{amsart}
\usepackage{setspace}
\usepackage{amsmath, amsfonts, amsthm, amstext, xspace, mathabx}
\usepackage[mathscr]{euscript}
\usepackage{amssymb,latexsym}
\usepackage{comment}
\usepackage{xcolor}
\usepackage{hyperref}
\definecolor{seagreen}{RGB}{46,139,87}
\definecolor{maroon}{RGB}{128,0,0}
\definecolor{darkviolet}{RGB}{148,0,211}
\definecolor{twelve}{RGB}{100,100,170}
\definecolor{thirteen}{RGB}{100,150,50}
\definecolor{fourteen}{RGB}{200,0,0}
\definecolor{fifteen}{RGB}{0,200,0}
\definecolor{sixteen}{RGB}{0,0,200}
\definecolor{seventeen}{RGB}{200,0,200}
\definecolor{eighteen}{RGB}{0,200,200}

\usepackage{centernot}
\usepackage{longtable}
\usepackage{mathtools}
\usepackage{stmaryrd}
\makeatletter
\newcommand{\xMapsto}[2][]{\ext@arrow 0599{\Mapstofill@}{#1}{#2}}
\def\Mapstofill@{\arrowfill@{\Mapstochar\Relbar}\Relbar\Rightarrow}
\makeatother

\newcommand{\chana}[1]{\begin{quote}{\color{violet}[\bf Hana: #1]} \end{quote}}
\newcommand{\hana}[1]{{\color{violet}{#1}}}

\newcommand{\dom}[1]{{\color{blue}{#1}}}

\newcommand{\jd}[1]{{\color{red}{#1}}}

\usepackage[all]{xy}
\newtheorem{thm}{Theorem}[section]
\newtheorem*{theorem*}{Theorem}
\newtheorem*{conjecture*}{Conjecture}
\newtheorem*{corollary*}{Corollary}
\newtheorem{lem}[thm]{Lemma}
\newtheorem{cor}[thm]{Corollary}

\newtheorem{prop}[thm]{Proposition}

\theoremstyle{definition}
\newtheorem{defin}[thm]{Definition}
\newtheorem{exm}[thm]{Example}
\newtheorem{rem2}[thm]{Remark}
\newtheorem*{rem2*}{Remark}

\newtheorem{ques}[thm]{Question}

\newtheorem{thmx}{Theorem}

\def\c{\mathbb{C}}

\def\f{\mathbb{F}}
\def\g{\mathbb{G}}

\def\m{\mathbb{M}}

\def\q{\mathbb{Q}}
\def\r{\mathbb{R}}
\def\s{\mathbb{S}}

\def\z{\mathbb{Z}}

\def\cd{\mathcal{D}}

\def\Spec{\operatorname{Spec}}

\def\chara{\operatorname{char}}

\def\Ext{\operatorname{Ext}}

\def\cofib{\operatorname{cofib}}

\def\re{\operatorname{Re}}
\def\SH{\operatorname{SH}}

\def\eff{\operatorname{eff}}
\def\veff{\operatorname{veff}}
\def\sc{\operatorname{sc}}
\def\vsc{\operatorname{vsc}}

\def\mASS{\operatorname{mASS}}

\def\aESSS{\operatorname{aESSS}}
\def\aVSSS{\operatorname{aVSSS}}
\renewcommand{\Re}{\operatorname{Re}}

\usepackage{amssymb}
\usepackage{tikz-cd}

\author{Dominic Culver}\address{Max-Planck-Institut f\"ur Mathematik}\email{dominic.culver@gmail.com}
\author{Hana Jia Kong}\address{The University of Chicago}\email{hanajk@math.uchicago.edu}
\author{J.D. Quigley}\address{Cornell University}\email{jdq27@cornell.edu}
\title{Algebraic slice spectral sequences}

\begin{document}
\maketitle

\begin{abstract}
For certain motivic spectra, we construct a square of spectral sequences relating the effective slice spectral sequence and the motivic Adams spectral sequence. We show the square can be constructed for connective algebraic K-theory, motivic Morava K-theory, and truncated motivic Brown--Peterson spectra. In these cases, we show that the $\r$-motivic effective slice spectral sequence is completely determined by the $\rho$-Bockstein spectral sequence. Using results of Heard, we also obtain applications to the Hill--Hopkins--Ravenel slice spectral sequences for connective Real K-theory, Real Morava K-theory, and truncated Real Brown--Peterson spectra. 
\end{abstract}

\tableofcontents

\section{Introduction} 

\subsection{Motivation and main theorems}
Two of the chief computational tools in motivic stable homotopy theory are the effective slice spectral sequence \cite{Lev08, SO12, Voe98} and the motivic Adams spectral sequence \cite{DI05,Mor99}. The purpose of this paper is to systematically relate the effective slice spectral sequence and the motivic Adams spectral sequence. 

The effective slice spectral sequence has many applications, such as a new proof of Milnor's Conjecture on quadratic forms \cite{RO16} and the calculation of the first stable homotopy group of motivic spheres \cite{RSO19} which identifies certain motivic stable stems with variants of K-theory. In general, the effective slice spectral sequence is an excellent tool for working over general base schemes because slices are often expressible in terms of arithmetic invariants like motivic cohomology. 

The motivic Adams spectral sequence is a powerful resource for making large-scale computations over specific base fields, such as the complex numbers \cite{IWX20}, the real numbers \cite{BI20}, and finite fields \cite{WO17}. Computations over $\Spec(\c)$ have brought about remarkable progress on the classical stable homotopy groups of spheres, and applications of $\r$-motivic computations to $C_2$-equivariant homotopy have been the topic of much recent work, such as \cite{BS20, BGI20, Kon20}. 

Our first main theorem fits the effective slice spectral sequence and motivic Adams spectral sequence into a square of spectral sequences.

\begin{thmx}[Comparison square, Theorem \ref{Thm:comparisonSquare}]\label{MainThm:comparisonSquare}
Let $E \in \SH(k)$ be a motivic spectrum which is algebraically $n$-sliceable over $k$ (Definition \ref{Def:AlgEff}). There is a square of spectral sequences of the form
\begin{equation}\label{Eqn:comparisonSquare}
\begin{tikzcd}
& \bigoplus_{q \geq 0} \Ext^{***}_{A}(H_{**}s_q^n E) \arrow[dl,Rightarrow,"n-aESSS"'] \arrow[dr,Rightarrow,"\bigoplus_{q \geq 0} mASS"] & \\
\Ext^{***}_A(H_{**}E) \arrow[dr,Rightarrow,"mASS"'] & & \bigoplus_{q \geq 0} \pi_{**}s_q^n E \arrow[dl,Rightarrow,"n-ESSS"] \\
& \pi_{**}E. 
\end{tikzcd}
\end{equation}
The spectral sequences in the square are discussed in Sections \ref{Sec:Background} and \ref{Sec:comparisonSquare}.
\end{thmx}

\begin{rem2*}
In this paper, we refer to a ``square of spectral sequences" whenever one can run the spectral sequences on one half of the square or the other half of the square  to attempt to compute the same groups. In particular, a square of spectral sequences is not the same as a commuting square in the category of spectral sequences of abelian groups, and we make no claims about convergence to the same filtration. 
\end{rem2*}

The condition that $E \in \SH(k)$ is algebraically $n$-sliceable is defined in terms of the homology of its slices and slice covers. Our second main theorem says that several familiar motivic spectra are algebraically $n$-sliceable.

\begin{thmx}[Theorems \ref{Thm:kglAlg}, \ref{Thm:knAlg}, and \ref{Thm:BPGLmSLiceable}, and Remark \ref{Rmk:kqsliceable}]\label{MainThm:kglknAlg}\hfill

Let $k$ be a perfect field. Theorem \ref{MainThm:comparisonSquare}\footnote{Or its very effective variant, see Remark \ref{Rmk:aVSSS}.} applies to the following motivic spectra:

\begin{enumerate}

\item The effective cover of algebraic K-theory, $kgl$, if the characteristic of $k$ is not two. 

\item For all $n \geq 1$, the connective motivic Morava K-theory, $k(n)$, if the exponential characteristic of $k$ is prime to the characteristic of the coefficients of $H_{**}$. 

\item For all $n \geq 1$, the $n$-th truncated motivic Brown--Peterson spectrum, $BPGL \langle n \rangle$, if $k$ has characteristic zero. 

\item The very effective cover of hermitian K-theory, $kq$, if the characteristic of $k$ is not two. 

\end{enumerate}
\end{thmx}

The $\rho$-Bockstein spectral sequence \cite{Hil11} is another important device in computational motivic stable homotopy theory. Theorems \ref{MainThm:comparisonSquare} and \ref{MainThm:kglknAlg} can be used to prove a precise relation between differentials in the $\rho$-Bockstein spectral sequence and the $\r$-motivic effective slice spectral sequence:

\begin{thmx}[Theorems \ref{Thm:RCorrespondence} and \ref{Thm:BPGLCorrespondence}]\label{MainThm:Differentials}
Let $k=\r$. There are 1-to-1 correspondences between the following differentials:
\begin{enumerate}
\item $\rho$-Bockstein spectral sequence $d_3$-differentials and effective slice spectral sequence $d_1$-differentials for $kgl$.
\item $\rho$-Bockstein spectral sequence $d_{2^{n+1}-1}$-differentials and type $n$ effective slice spectral sequence $d_1$-differentials for $k(n)$. 
\item For each $1 \leq i \leq n$, $\rho$-Bockstein spectral sequence $d_{2^{i+1}-1}$-differentials and effective slice spectral sequence $d_{2^i-1}$-differentials for $BPGL \langle n \rangle$. 
\end{enumerate}
Moreover, the (type $n$) effective slice spectral sequence differentials are completely determined by knowledge of the corresponding $\rho$-Bockstein spectral sequence differentials.
\end{thmx}

Formal similarities between the $\rho$-Bockstein and effective slice spectral sequence differentials have been observed in previous work, such as \cite[Sec. 6]{Hea19}. Theorem \ref{MainThm:Differentials} and its proof may be viewed as a rigorous explanation for these similarities. 

Finally, our computations have consequences in the $C_2$-equivariant stable homotopy theory. The Hill--Hopkins--Ravenel slice spectral sequence \cite{HHR16} has enjoyed many applications since its use in the solution to the Kervaire invariant one problem, and thus we are motivated to calculate differentials for common $C_2$-spectra. In previous work, Heard \cite{Hea19} compared the $\r$-motivic effective slice filtration and the Hill--Hopkins--Ravenel slice filtration \cite{HHR16} and deduced certain differentials in the effective slice spectral sequence using previously computed differentials in the HHR slice spectral sequence. Using Heard's comparison results and Theorem \ref{MainThm:Differentials}, we reverse this logic to produce differentials in the HHR slice spectral sequences for the genuine $C_2$-spectra $k\r$, $k\r(n)$, and $BP\r\langle n \rangle$ from known slice spectral sequence differentials. Our results recover work of Dugger \cite{Dug05} and Hill--Hopkins--Ravenel \cite{HHR16}. 

\begin{thmx}[Corollaries \ref{Cor:kRdiff}, \ref{Cor:kndiff}, and \ref{Cor:BPRdiff}]\label{MainThm:C2}
The following statements hold:

\begin{enumerate}

\item The nontrivial differentials in the HHR slice spectral sequence for $\pi_{**}^{C_2}k\r$ are determined via Leibniz rule by
$$d_3(\tau^2) = \rho^3 \bar{v}_1 \quad \text{ and } \quad d_3\left(\frac{\gamma}{\rho^3 \tau^2} \right) = \frac{\gamma}{\tau^4} \bar{v}_1.$$

\item Let $n \geq 1$. There is a nontrivial differential in the HHR slice spectral sequence for $\pi_{**}^{C_2}k\r(n)$ of the form $d_{2^{n+1}-1}(\tau^{2^{n+1}}) = \rho^{2^{n+1}-1}\bar{v}_n$. 

\item Let $n \geq 1$. There are nontrivial differentials in the HHR slice spectral sequence for $\pi_{**}^{C_2}BP\r\langle n \rangle$ of the form
$d_{2^{i+1}-1}(\tau^{2^{i+1}}) = \rho^{2^{i+1}-1}\bar{v}_i$, $1 \leq i \leq n$.

\end{enumerate}

The relevant elements are recalled in Remark \ref{Rem:namingconvention}. 
\end{thmx}

\subsection{Organization}

In Section \ref{Sec:Background}, we recall the functors and spectral sequences discussed in this paper, such as (equivariant) Betti realization, variants of the effective and very effective slice spectral sequences, and the motivic Adams spectral sequence. 

In Section \ref{Sec:comparisonSquare}, we define when a motivic spectrum is algebraically $n$-sliceable and prove Theorem \ref{MainThm:comparisonSquare}. We then analyze several examples and non-examples of algebraically $n$-sliceable spectra, and in particular we prove Theorem \ref{MainThm:kglknAlg}.

In Section \ref{Sec:Differentials}, we recall the $\rho$-Bockstein spectral sequence and differentials in the motivic Adams spectral sequence; we then prove Theorem \ref{MainThm:Differentials}.

In Section \ref{Sec:C2}, we recall Heard's comparison of the motivic and $C_2$-equivariant slice filtrations and apply it to prove Theorem \ref{MainThm:C2}.

\subsection{Conventions}
\label{subsection:conventions}
\begin{enumerate}
\item After Section \ref{Sec:Background}, everything is implicitly $(p,\eta)$-complete.\footnote{This agrees with $p$-completion in many cases we consider. See e.g. \cite[Thm. 1]{HKO11a}.} 
\item $k$ is a perfect field of expontential characteristic prime to $p$. 
\item $\SH$ denotes the classical stable homotopy category, $\SH(k)$ denotes the motivic stable homotopy category over $\Spec(k)$, and $\SH^{C_2}$ denotes the $C_2$-equivariant stable homotopy category. 
\item $H$ denotes the classical, equivariant, or motivic mod $p$ Eilenberg-MacLane spectrum. 
\item $\m^k_p$ denotes the mod $p$ motivic cohomology of $\Spec(k)$. 
\item $A$ (resp. $A^\vee$) denotes the motivic Steenrod algebra (resp. its $\m_p^k$-linear dual).
\item If $M$ is an $A$-module, $M^\vee$ denotes its $\m_p^k$-linear dual. 
\item The homotopy groups of $E \in \SH(k)$ are denoted $\pi_{m,u}^k(E)$, or if there is no confusion about the base field, $\pi_{m,u}(E)$. Our bigrading convention is that $\pi_{m,u}^k(E) \cong \SH(k)[\Sigma^\infty (S^1_\bullet)^{\wedge m-u} \wedge \g_m^{\wedge u}, E]$. 
\item The homotopy groups of $E \in \SH^{C_2}$ are denoted $\pi_{m,u}^{C_2}(E)$. Our bigrading convention is that $\pi_{m,u}^{C_2}(E) \cong \SH^{C_2}[\Sigma^\infty S^{m-u} \wedge S^{u\sigma}, E]$ where $\sigma$ is the one-dimensional real sign representation of $C_2$. 
\item We will often write $\Ext_A^{***}(E)$ in place of $\Ext_A^{***}(H_{**}E)$. 
\item We use the following abbreviations for spectral sequences:
\begin{enumerate}
\item mASS: motivic Adams spectral sequence.
\item ESSS and VSSS: effective and very effective slice spectral sequence.
\item aESSS and aVSSS: algebraic effective and very effective slice spectral sequence.
\item $n$-(-)SSS: type $n$ (-)SSS, with (-) either E, V, aE, or aV. 
\item $H_*$-$n$-(-)SSS: homological type $n$ (-)SSS, with (-) as above.
\item $\rho$-BSS: $\rho$-Bockstein spectral sequence.
\end{enumerate} 
\item Notation for $C_2$-equivariant homotopy elements is contained in Table \ref{Table:C2}.

\begin{table}
\begin{tabular}{ | c | c | c | c || c | c | c | c | }
 \hline
 \multicolumn{4}{|c|}{$\pi_{**}^{C_2}H\f_2$} & \multicolumn{4}{|c|}{$\pi_{**}^{C_2}H\z_2$} \\
 \hline
 element & alt. name & bidegree & $RO(C_2)$-degree & element & alt. name & bidegree & $RO(C_2)$-degree \\
 \hline
$\rho$ & $a_\sigma$ & $(-1,-1)$ & $-\sigma$ & $\rho$ & $a_\sigma$ & $(-1,-1)$ & $-\sigma$ \\
$\tau$ & $u_\sigma$ & $(0,-1)$ & $1-\sigma$ & $\tau^2$ & $u_{2\sigma}$ & $(0,-2)$ &  $2-2\sigma$\\
$\frac{\gamma}{\tau}$& $\theta$ & $(0,2)$  & $2\sigma-2$ & $\frac{\gamma}{\tau}$ & $2u_{2\sigma}^{-1}$ & $(0,2)$ & $2\sigma-2$\\
	& 			& 			& 	& $\frac{\gamma}{\tau^2}$ & $\theta$ & $(0,3)$ & $3\sigma-3$\\
 \hline
\end{tabular}
\caption{Names, alternate names, and degrees of elements in $C_2$-equivariant homotopy groups.}\label{Table:C2}
\end{table}

\end{enumerate}

\subsection{Acknowledgements}

The authors thank the organizers of the BIRS workshop ``Equivariant stable homotopy theory and $p$-adic Hodge theory," where some early stages of this work were completed. The authors thank Mark Behrens, Bob Bruner, Drew Heard, Dan Isaksen, XiaoLin Danny Shi, Inna Zakharevich, and an anonymous referee for helpful comments and questions. Some of the work on this paper was done while the first author was at the Max Planck Institute for Mathematics. He would like to thank the Institute for their support. 

\section{Motivic slice and Adams spectral sequences}\label{Sec:Background}

We start by discussing the functors and spectral sequences which we will study in the rest of the paper.

\subsection{(Equivariant) Betti realization functors}\label{SS:Betti}

Since we will mention them in our discussion of the various spectral sequences below, we start by recalling the Betti realization and equivariant Betti realization functors which relate the motivic, classical, and $C_2$-equivariant stable homotopy categories. 

Suppose the base field $k$ has a complex embedding. The Betti realization functor \cite{MV99}
\begin{equation}\label{Eqn:Betti}
\Re: \SH(k) \to \SH
\end{equation}
is defined by sending $X \in \SH(k)$ to $X(\c) \in \SH$. Similarly, if $k$ has a real embedding, the equivariant Betti realization functor \cite{MV99}
\begin{equation}\label{Eqn:EBetti}
\Re_{C_2}: \SH(k) \to \SH^{C_2}
\end{equation}
is defined by sending $X \in \SH(k)$ to $X(\c) \in \SH^{C_2}$, where $C_2$ acts on $X(\c)$ via complex conjugation. 

\subsection{The type $n$ effective slice spectral sequence}\label{SS:EffSSS}

We now introduce one of our main objects of study, the type $n$ effective slice spectral sequence, which generalizes the effective slice spectral sequence. Throughout this section, we fix a prime $p$. 

Let $\SH^{\eff}(k)$ be the localizing subcategory of $\SH(k)$ generated by $\{ \Sigma^\infty_+ X : X \in Sm_k\}$. These subcategories form the \emph{effective slice filtration}
\begin{equation}\label{Eqn:EffFiltSH}
\cdots \subset \Sigma^{2q+2,q+1} \SH^{\eff}(k) \subset \Sigma^{2q,q} \SH^{\eff}(k) \subset \Sigma^{2q-2,q-1} \SH^{\eff}(k) \subset \cdots.
\end{equation}
The inclusion $i_q: \Sigma^{2q,q}\SH^{\eff}(k) \subset \SH(k)$ is left adjoint to the functor $r_q: \SH(k) \to \Sigma^{2q,q}\SH^{\eff}(k)$. 

\begin{defin}
The composite $f_q:=i_q \circ r_q$ is the \emph{$q$-th effective slice cover} and $s_q := \cofib(f_{q+1} \to f_q)$ is the \emph{$q$-th effective slice}. 

For each $n \geq 1$, we define the \emph{$q$-th type $n$ effective slice cover} $f_q^n := f_{q(p^n-1)}$ and the \emph{$q$-th type $n$ effective slice} $s_q^n := \cofib(f_{q+1}^n \to f_q^n)$. Observe that $f_q^1 = f_q$ if $p=2$. 
\end{defin}

For any motivic spectrum $E \in \SH(k)$ and any integer $n \geq 1$, we obtain a $\z$-indexed tower of fibrations
\begin{equation}\label{Eqn:EffTower}
\begin{tikzcd}
\vdots \ar[d] & \\
f_{q+1}^nE \ar[d]\ar[r] & s_{q+1}^nE \\
f_q^nE \ar[d]\ar[r] & s_q^nE\\
f_{q-1}^nE \ar[d] \ar[r] & s_{q-1}^nE \\
\vdots 
\end{tikzcd}
\end{equation}
called the \emph{type $n$ effective slice tower}. Applying motivic homotopy groups gives rise to the \emph{type $n$ effective slice spectral sequence}
\begin{equation}\label{Eqn:EffSSS}
{^{n,\eff}E}_{m,q,u}^1(E) = \pi_{m,u}s_q^nE \Rightarrow \pi_{m,u}\sc(E)
\end{equation}
with differentials $d^r(E) : {^{n,\eff}E}^r_{m,q,u}(E) \to {^{n,\eff}E}^{r}_{m-1,q+r,u}(E)$. 
Let $f^q_nE$ denote the cofiber of $f_{q+1}^n E\to E.$
The term $\sc(E)$ in the abutment is the \emph{slice completion} of $E$ defined as the homotopy inverse limit $\sc(E) := \lim_q f_q E$. 
For each $n$, we have $\sc(E) \simeq \lim_q f^q_nE$ since the type $n$ effective slice filtration is a sped up version of the ordinary effective slice filtration. 

For each $E \in \SH(k)$ we study below, there is an equivalence $\sc(E) \simeq E$ by \cite[Thm. 8.12]{Hoy15}. More generally, \cite[Thm. 3.50]{RSO19} implies that in very general situations, $\sc(E)$ may be identified with the $\eta$-completion of $E$. 

\subsection{The type $n$ very effective slice spectral sequence}\label{SS:VeffSSS}

We now discuss a variant of the type $n$ effective slice spectral sequence. Throughout this section, fix an integer $n \geq 1$ and prime $p$. 

Let $\SH^{\veff}(k)$ be the smallest full subcategory of $\SH(k)$ that contains all suspension spectra of smooth $k$-schemes of finite type and which is closed under all homotopy colimits and extensions. Repeating the construction from Section \ref{SS:EffSSS} with $\SH^{\veff}(k)$ in place of $\SH^{\eff}(k)$ yields the \emph{$q$-th type $n$ very effective slice cover} $\tilde{f}_q^n$ and the \emph{$q$-th type $n$ very effective slice} $\tilde{s}_q^n$. These fit into a tower analogous to \eqref{Eqn:EffTower} which gives rise to the \emph{type $n$ very effective slice spectral sequence}
\begin{equation}\label{Eqn:VeffSSS}
^{n,\veff}E_{m,q,u}^1(E) = \pi_{m,u}\tilde{s}_q^nE \Rightarrow \pi_{m,u}\vsc(E) \cong \pi_{m,u} E
\end{equation}
with differentials $d^r(E) : {^{n,\veff}E}^r_{p,q,u}(E) \to {^{n,\veff}E}^r_{p-1,q+r,u}(E)$. If we let $\tilde{f}^q_n E$ denote the cofiber of the map $\tilde{f}_{q+1}^n E\to E$, then the \emph{very effective slice completion} $\vsc(E) := \lim_q \tilde{f}^q_1 E\simeq \lim_q \tilde{f}^q_n E$ is equivalent to $E$ because $\tilde{s}_q^n E$ is $q(p^n-1)$-connected in Morel's homotopy t-structure on $\SH(k)$ for all $q \in \z$.\footnote{See \cite[Sec. 5]{SO12} for a full discussion of convergence of the very effective slice spectral sequence.} 

\begin{rem2}\label{Rmk:VESSPostnikov}
Suppose the ground field $k$ admits a complex embedding. Then the very effective slice spectral sequence over $\Spec(k)$ is compatible with the double-speed Postnikov filtration under Betti realization \cite{GRSO12}, so there is a map of spectral sequences from the very effective slice spectral sequence for $X \in \SH(k)$ to the double-speed Postnikov spectral sequence of $\Re(X)$. 
\end{rem2}

\begin{rem2}\label{Rmk:VESSHHR}
Suppose the ground field $k$ admits a real embedding. Then the very effective slice spectral sequence over $\Spec(k)$ for a localized quotient of $MGL$ is compatible with the Hill--Hopkins--Ravenel slice spectral sequence \cite{HHR16} of its equivariant Betti realization by \cite[Thms. 5.15-5.16]{Hea19}. 
\end{rem2}

\begin{rem2}\label{Rmk:Coincide}
The effective and very effective slice filtrations coincide for localized quotients of $MGL$ \cite[Prop. 4.3]{Hea19} and Landweber exact motivic spectra \cite[Prop. 4.11]{Hea19}.
\end{rem2}

\subsection{The motivic Adams spectral sequence}\label{SS:mASS}

Let $E \in \SH(k)$ be a unital motivic spectrum and let $\overline{E} := \cofib(S^{0,0} \to E)$. The \emph{canonical $E$-based Adams resolution} of the sphere spectrum $\s$ is given by
\begin{equation}\label{Eqn:EAdamsResolution}
\begin{tikzcd}
\s \ar[d] & \Sigma^{-1,0} \overline{E} \ar[l]\ar[d] & \Sigma^{-2,0} \overline{E}^{\wedge 2} \ar[l]\ar[d] & \cdots \ar[l] \\
E & \Sigma^{-1,0} E \wedge \overline{E} & \Sigma^{-2,0} E \wedge \overline{E}^{\wedge 2}. & 
\end{tikzcd}
\end{equation}
The canonical $E$-based Adams resolution over a motivic spectrum $X \in \SH(k)$ is obtained by smashing every term in  \eqref{Eqn:EAdamsResolution} with $X$. 

When $E = H$ is mod $p$ motivic cohomology, we have $H_{**}(S^{0,0}) = \m_p^k$ and $\pi_{**}(H \wedge H) = A^\vee$ is the dual motivic Steenrod algebra \cite{HKO17, Voe03}. Applying motivic homotopy groups to \eqref{Eqn:EAdamsResolution} gives the \emph{motivic Adams spectral sequence}
\begin{equation}\label{Eqn:mASS}
{^{\mASS}E}_2^{s,t,u} = \Ext_{A^\vee}^{s,t,u}( \m^k_p, H_{**}(X)) \Rightarrow \pi_{t-s,u}(X_H^\wedge)
\end{equation}
with $d_r : {^{\mASS}E}_r^{s,t,u} \to {}^{\mASS}E_r^{s+r,t+r-1,u}$. Here, $X_H^{\wedge}$ is the \emph{$H$-nilpotent completion} of $E$. By \cite[Thm. 1]{HKO11a}\cite[Cor. 6.1]{KW19}\cite[Thm. 1.0.1]{Man18}, there is an equivalence $X_H^\wedge \simeq X_{(p,\eta)}^\wedge$ for $X$ sufficiently nice\footnote{In this paper, ``sufficiently nice" includes all motivic spectra $X$ which are connective and cellular of finite type.}.  Moreover, if $k$ is a field of characterstic zero, and if $p>2$ and $\cd_p(k) < \infty$ and $-1 \in k$ is a sum of squares, or if $p=2$ and $k$ has finite virtucal cohomological dimension, then $X_{(p,\eta)}^\wedge \simeq X_p^\wedge$ \cite{HKO11a}. 

\begin{rem2}
The motivic homology $\m_p^k$ is reviewed in Section \ref{SS:RhoBSS}.
\end{rem2}

\begin{rem2}
As in Remark \ref{Rmk:VESSPostnikov}, if the ground field $k$ admits a complex embedding, the motivic Adams and Adams--Novikov filtrations of $X \in \SH(k)$ are compatible with the Adams and Adams--Novikov filtrations of $\Re(X) \in \SH$. Thus Betti realization induces a map from the motivic Adams(--Novikov) spectral sequence to the classical Adams(--Novikov) spectral sequence. 

Similarly (c.f. Remark \ref{Rmk:VESSHHR}), if the ground field $k$ admits a real embedding, then the motivic Adams and Adams--Novikov filtrations of $X \in \SH(k)$ are compatible with the $C_2$-equivariant Adams and Adams--Novikov filtrations of $\Re_{C_2}(X) \in \SH^{C_2}$. Thus equivariant Betti realization induces a map between the corresponding spectral sequences. 
\end{rem2}

\section{Comparison square and examples}\label{Sec:comparisonSquare}

In this section, we introduce a condition on a motivic spectrum which allows for the construction of a comparison square relating the effective slice spectral sequence and the motivic Adams spectral sequence. We prove that many common motivic spectra satisfy this condition, and we find an example which does not. From now on, we will implicitly $p$-complete everywhere. 

\subsection{Algebraic slice spectral sequences}\label{SS:AlgSSS}

The type $n$ effective slice tower \eqref{Eqn:EffTower} is built out of cofiber sequences
$$f_{q+1}^nE \to f_q^n E \to s_q^n E.$$
Applying mod $p$ motivic homology induces a long exact sequence
\begin{equation}\label{Eqn:HLES}
\cdots \to H_{i,j}f_{q+1}^nE \xrightarrow{ } H_{i,j}f_q^n E \xrightarrow{ } H_{i,j} s_q^n E \xrightarrow{} H_{i-1,j} f_{q+1}^n E \to \cdots
\end{equation}
for each $j \in \z$. 

\begin{defin}\label{Def:AlgEff}
We say that $E \in \SH(k)$ is \emph{algebraically $n$-sliceable over $k$} if the long exact sequences \eqref{Eqn:HLES} split into short exact sequences of $A^\vee$-comodules
\begin{equation}\label{Eqn:HSES}
0 \xrightarrow{ } H_{**}f_q^nE \xrightarrow{p} H_{**}s_q^nE \xrightarrow{\iota}H_{*-1,*}(f_{q+1}^nE)\to 0
\end{equation}
for each $q \in \z$. 

Similarly, we say that $E$ is \emph{algebraically sliceable over $k$} if the long exact sequences associated to the effective slice tower split into short exact sequences of $A^\vee$-comodules
\begin{equation}\label{Eqn:HSES}
0 \xrightarrow{ } H_{**}f_q E \xrightarrow{p} H_{**}s_q E \xrightarrow{\iota}H_{*-1,*}(f_{q+1} E)\to 0
\end{equation}
for each $q \in \z$. 

When the field $k$ is clear, we will omit the phrase ``over $k$" for brevity.

\end{defin}

\begin{rem2}
If $p=2$, being algebraically sliceable is the same as being algebraically $1$-sliceable. 
\end{rem2}

%

Applying the functor $\Ext_{A^\vee}^{***}(\m_p^k, -)$ to \eqref{Eqn:HSES} gives rise to long exact sequences which assemble into the unrolled exact couple
\begin{equation}\label{Eqn:Unrolled}
\begin{tikzcd}
	\cdots & \Ext^{s,t,u}(H_{**}f_0^nE)\arrow[l,swap,"\delta"] \arrow[d,"p_*"]& \Ext^{s-1, t-1,u}(H_{**}f_1^nE)\arrow[d,"p_*"]\arrow[l,swap,"\delta"] & \Ext^{s-2, t-2, u}(H_{**}f_2^nE)\arrow[d,"p_*"]\arrow[l,swap,"\delta"] & \arrow[l,swap,"\delta"]\cdots\\
	\mbox{}\arrow[ur, dotted,"\iota_*"]& \Ext^{s,t,u}(H_{**}s_0^nE)\arrow[ur, dotted,"\iota_*"] & \Ext^{s-1,t-1,u}(H_{**}s_1^nE)\arrow[ur,dotted,"\iota_*"] & \Ext^{s-2,t-2,u}(H_{**}s_2^nE)\arrow[ur, dotted,"\iota_*"] & 
\end{tikzcd}
\end{equation}
in which the dotted arrows have tridegree $(0,-1,0)$. This gives rise to the \emph{algebraic type $n$ effective slice spectral sequence}
\begin{equation}
{^{n,\aESSS}E}^{q,s,t,u}_1(E) = \Ext_{A^\vee}^{s,t,u}( \m^k_p, H_{**}s_q^nE) 
\end{equation}
with differentials 
\[
	d_r: E_r^{q,s,t,u}\to E_r^{q+r,s-r+1, t-r, u}.
\] 

\begin{prop}
The spectral sequence \eqref{Eqn:AlgSSS} converges conditionally to the Ext groups of $H_{**}(E)$:
\begin{equation}\label{Eqn:AlgSSS}
{^{n,\aESSS}E}^{q,s,t,u}_1(E)\implies \Ext^{s+q,t+q,u}(H_{**}E).
\end{equation}
Moreover, if $E$ is effective, i.e. $E = f_0E$, and the spectral sequence collapses at a finite stage, then the spectral sequence converges strongly. 
\end{prop}

\begin{proof}
We will verify that the spectral sequence \eqref{Eqn:AlgSSS} converges conditionally to the colimit $\Ext^{***}(H_{**}(E))$ in the sense of Boardman \cite[Def. 5.10]{Boa98}: for each fixed tridegree $(s,t,u)$, we will show that
\begin{equation}\label{Eqn:lim}
\lim_k \Ext^{s-k,t-k,u}(H_{**}f_k^n E) = 0,
\end{equation}
\begin{equation}\label{Eqn:lim1}
{\lim_k}^1 \Ext^{s-k,t-k,u}(H_{**}f_k^n E) = 0.
\end{equation}
The limits above are taken along the maps $\delta$ in the unrolled exact couple \label{Eqn:Unrolled}.

Fix a tridegree $(s,t,u)$. Observe that 
$$\Ext^{s-k,t-k,u}(H_{**}f_k^n E) = 0$$
whenever $k > s$, so \eqref{Eqn:lim} is clear. For \eqref{Eqn:lim1}, note that 
$$\Ext^{s-k,t-k,u}(H_{**}s_k^n E) = 0$$
for all $k>s$, so
$$p_* : \Ext^{s-k, t-k, u}(H_{**}f_k^n E) \to \Ext^{s-k,t-k,u}(H_{**}s_k^n E)$$
is zero for all $k > s$ and thus
$$\delta : \Ext^{s-(k+1),t-(k+1),u}(H_{**} f_{k+1}^n E) \to \Ext^{s-k,t-k,u}(H_{**}f_k^n E)$$
is surjective for all $k > s$. Therefore the inverse system $(\Ext^{s-k,t-k,u}(H_{**}f_k^nE), \delta)$ is Mittag-Leffler, cf. \cite[Pg. 9]{Boa98}, and the equality \eqref{Eqn:lim1} holds. 

We have now shown that \eqref{Eqn:AlgSSS} converges conditionally to the colimit. We now justify strong convergence under the hypotheses that $E$ is effective and the spectral sequence collapses at a finite stage. Since $E$ is effective, \eqref{Eqn:AlgSSS} is a half-plane spectral sequence with entering differentials \cite[Pg. 20]{Boa98}.\footnote{More precisely, the restriction to each fixed weight $u$ is a half-plane spectral sequence with entering differentials.}  By \cite[Remark, Pg. 20]{Boa98}, we have $RE_\infty=0$ whenever the spectral sequence collapses at a finite stage. The spectral sequence therefore converges strongly under our hypotheses by \cite[Thm. 7.1]{Boa98}. 
\end{proof}

The same discussion carries over to the type $n$ very effective slice filtration. In particular, we can define when $E \in \SH(k)$ is \emph{very algebraically $n$-sliceable over $k$}, and for such $E$ we may define the \emph{algebraic type $n$ very effective slice spectral sequence}
\begin{equation}\label{Eqn:AlgVSSS}
{^{n,\aVSSS}E}^{q,s,t,u}_1(E) = \Ext_{A^\vee}^{s,t,u}( \m^k_p, H_{**}\tilde{s}_q^nE) \Rightarrow \Ext_{A^\vee}^{s,t,u}( \m^k_p, H_{**}\sc(E)).
\end{equation}
The spectral sequence converges conditionally. If $E$ is very effective and the spectral sequence collapses at a finite stage, the spectral sequence converges strongly.

\subsection{Comparison square}\label{SS:comparisonSquare}

Our main tool for relating the type $n$ effective slice spectral sequence and the motivic Adams spectral sequence is the following theorem, which assembles the spectral sequences \eqref{Eqn:EffSSS}, \eqref{Eqn:mASS}, and \eqref{Eqn:AlgSSS}  discussed above into a square. 

\begin{thm}[Comparison square]\label{Thm:comparisonSquare} 
Let $E \in \SH(k)$ be a $k$-motivic spectrum. If $E$ is algebraically $n$-sliceable, then there is a square of spectral sequences of the form
\begin{equation}\label{Eqn:comparisonSquare}
\begin{tikzcd}
& \bigoplus_{q \geq 0} \Ext^{s,f,u}_{A}(H_{**}s_q^n E) \arrow[dl,Rightarrow,"n-aESSS"'] \arrow[dr,Rightarrow,"\bigoplus_{q \geq 0} mASS"] & \\
\Ext^{s+q,f+q,u}_A(H_{**}E) \arrow[dr,Rightarrow,"mASS"'] & & \bigoplus_{q \geq 0} \pi_{s,u}s_q^n E \arrow[dl,Rightarrow,"n-ESSS"] \\
& \pi_{s+q,u}E. 
\end{tikzcd}
\end{equation}
\end{thm}

\begin{rem2}\label{Rmk:aVSSS}
The obvious analog holds if we work with the type $n$ very effective slice filtration instead of the type $n$ effective slice filtration. Similarly, we may work with the ordinary (very) effective slice filtration instead of the type $n$ (very) effective slice filtration. 
\end{rem2}

\subsection{Examples}\label{SS:Examples}

We now present some examples of motivic spectra which are algebraically sliceable.

\subsubsection{Zero slices}

We begin with a simple example.

\begin{exm}
If $E \in \SH(k)$ is effective and the canonical map $f_0^n E \to s_0^nE$ is an equivalence, then the type $n$ effective slice tower of $E$ is just an equivalence at the bottom level and zero elsewhere. The only nontrivial long exact sequence in homology then splits into a short exact sequence of the form
$$0 \to H_{**}f_0^n E \xrightarrow{\cong} H_{**}s_0^n E \to 0 \to 0,$$
so $E$ is algebraically $n$-sliceable. This applies, for example, if $E = H\f_p$ or $E = H\z_p$, where $n$ can be any positive integer, and $k$ can be any base field. 
\end{exm}

\subsubsection{Connective algebraic K-theory}

Most spectra of interest are not concentrated in a single type $n$ slice. However, the following example shows that some more complicated spectra are algebraically sliceable. 

\begin{thm}\label{Thm:kglAlg}
Let $p=2$ and let $k$ be any base field of characteristic not two. The effective cover of algebraic K-theory, $kgl$, is algebraically and very algebraically sliceable over $k$.
\end{thm}

\begin{proof}
The effective and very effective slice filtrations of $kgl$ coincide by \cite[Prop. 4.3]{Hea19} since it is a localized quotient of $MGL$, so it suffices to prove that $kgl$ is algebraically sliceable. Recall that for any $E \in \SH(k)$, we have by \cite[Lem. 8]{Bac17} that $\Sigma^{2,1} f_n E\simeq f_{n+1}(\Sigma^{2,1}E)$. Bott periodicity for algebraic K-theory \cite[Sec. 6.2]{Voe98a} implies that $\Sigma^{2,1}KGL \simeq KGL$, so $\Sigma^{2,1} f_n KGL \simeq f_{n+1}KGL$ for all $n \in \z$. Therefore
\begin{equation}
f_n kgl \simeq f_n f_0 KGL \simeq f_n KGL \simeq \Sigma^{2n,n} f_0 KGL \simeq \Sigma^{2n,n} kgl
\end{equation}
for all $n \geq 0$ since $kgl = f_0 KGL$. We also recall that $s_q kgl \simeq \Sigma^{2q,q} H\z_p$ if $q \geq 0$ and $s_q kgl = 0$ if $q < 0$. 

The cofiber sequences of motivic spectra
$$f_q kgl \to s_q kgl \to \Sigma^{1,0} f_{q+1} kgl$$
thus can be rewritten as
$$\Sigma^{2q,q} kgl \to \Sigma^{2q,q}H\z_p \to \Sigma^{2q+3,q+1} kgl.$$
Applying homology gives a short exact sequence
$$0 \to \Sigma^{2q,q}(A//E(1))^\vee \xrightarrow{i} \Sigma^{2q,q}(A//E(0))^\vee \xrightarrow{j} \Sigma^{2q+3,q+1}(A//E(1))^\vee \to 0$$
where $i$ is the evident inclusion of $A^\vee$-comodules and $j$ is the evident projection of $A^\vee$-comodules. This follows from the identification of $f_1 kgl \to f_0 kgl$ with $v_1 : \Sigma^{2,1} kgl \to kgl$, along with the fact that $H_{**}(v_1) = 0$. It follows that $kgl$ is algebraically sliceable. 
\end{proof}

\subsubsection{Connective motivic Morava K-theory}

We now include an example to demonstrate the necessity of the type $n$ effective slice filtration. Throughout this section, let $p$ be a prime and $k$ a perfect field of exponential characteristic prime to $p$. 

\begin{prop}\label{Prop:knNot}
Let $k(n)$ denote the $n$-th connective motivic Morava K-theory \cite[Def. 4.3]{LT15}.\footnote{The definition of $k(n)$ given by Levine--Tripathi in \cite{LT15} as a localized quotient of algebraic cobordism works for all base fields with exponential characteristic prime to $p$. This definition generalizes the earlier definition of the auxiliary spectrum $k'(n)$ given by Borghesi over fields with a complex embedding \cite[Sec. 4.2]{Bor03} and certain perfect fields \cite[Sec. 4]{Bor09}.}  If $n>1$, then $k(n)$ is not algebraically or very algebraically sliceable. 
\end{prop}

\begin{proof}
The effective and very effective slice filtrations of $k(n)$ coincide by \cite[Prop. 4.3]{Hea19} since it is a localized quotient of $MGL$, so it suffices to prove that $k(n)$ is not algebraically sliceable. By \cite[Cor. 4.7]{LT15}, we have 
\begin{equation}\label{Eqn:knslice}
	s_q k(n) \simeq \begin{cases}
		\Sigma^{2q,q}H\f_p \quad & \text{ if } q\equiv 0 \mod (p^n-1),\\
		* \quad & \text{ if else.}
	\end{cases}
\end{equation}
In particular, we have $s_1k(n) = 0$ and  $f_1k(n) \simeq f_2k(n)$ whenever $n>1$. Applying homology of cofiber sequence
$$f_1 k(n) \to s_1 k(n) \to \Sigma^{1,0}f_2k(n)$$
therefore yields
$$A^\vee \to 0 \to \Sigma^{1,0} A^\vee$$
which cannot be a short exact sequence since $A^\vee$ is bounded below in each fixed motivic weight. Thus $k(n)$ is not algebraically sliceable.  
\end{proof}

Although $k(n)$ is not algebraically sliceable if $n>1$, the arguments showing $kgl$ is algebraically sliceable extend to show that $k(1)$ is algebraically sliceable: the key point is that $s_qk(1) \neq *$ for all $q \geq 0$, and thus we avoid the issues in the previous proof. More generally, the type $n$ effective slice filtration allows us to avoid these trivial slices:

\begin{thm}\label{Thm:knAlg}
Let $n$ be a positive integer. The $n$-th connective motivic Morava K-theory, $k(n)$, is algebraically and very algebraically $n$-sliceable. 
\end{thm}

\begin{proof}
It suffices to prove $k(n)$ is algebraically $n$-sliceable. The slices of $k(n)$ are given by Equation \eqref{Eqn:knslice}, so we have
$$s_q^n k(n) \simeq \Sigma^{q\cdot (2p^{n}-2), q \cdot (p^{n}-1)} H\f_p$$
for all $q \geq 0$. We also observe that by periodicity of $k(n)$, we have 
$$f_{q+1}^nk(n) \simeq \Sigma^{2p^{n}-2,(p^{n}-1)} f_q^n k(n) \simeq \cdots \simeq \Sigma^{(q+1)(2p^{n}-2),(q+1)(p^{n}-1)} k(n).$$

By \cite[Lem. 6.10]{Hoy15}, the cohomology of $k(n)$ is $A//E(Q_n)$, where $E(Q_n)$ is the sub-Hopf algebroid of motivic Steenrod algebra generated by the motivic Milnor primitive $Q_n$.
We are now in a situation mirroring that of Theorem \ref{Thm:kglAlg}. Just like in that proof, the long exact sequences in homology associated to the cofiber sequences
$$f_{q+1}^n k(n) \to f_q^n k(n) \to s_q^n k(n)$$
split into short exact sequences of $A^\vee$-comodules of the form
\begin{align*}
0 \to \Sigma^{q(2p^{n}-2),q(p^{n}-1)}(A//E(Q_n))^\vee &\xrightarrow{i} \Sigma^{q(2p^{n}-2),q(p^{n}-1)} A^\vee \\
&\xrightarrow{j} \Sigma^{(q+1)(2p^{n}-2)+1,(q+1)(p^{n}-1)} (A//E(Q_n))^\vee \to 0
\end{align*}
since $H_{**}(v_n) = 0$. Therefore $k(n)$ is algebraically $n$-sliceable. 
\end{proof}

\subsubsection{Truncated motivic Brown--Peterson spectra}

Like $k(n)$, the truncated motivic Brown--Peterson spectrum $BPGL\langle m \rangle$ is not algebraically $1$-sliceable over certain base fields when $m>1$, but it is algebraically $m$-sliceable over certain base fields. The proof is more subtle in this case, though.

\begin{lem}
\label{Lem:lemma}
	Consider a tower of fibrations
	$$
	\begin{tikzcd}
&\vdots \ar[d] & \\
&f_2E\ar[d]\ar[r] &s_2E\\
&f_1E\ar[d]\ar[r] &s_1E\\
E&f_0E \ar[l,"\simeq"']\ar[r] & s_0E \\
\end{tikzcd}$$
where $f_qE \to f_{q-1} E\to s_{q-1}E$ is a fibration for each $q\geq 1$. If $H_*(f_qE\to f_{q-1}E)$ is zero for each $q\geq 1$, then the associated spectral sequence
$$E_1^{**}=H_*(s_*E)\implies H_*E$$ collapses at the $E_2$-page.
\end{lem}

\begin{proof}
	Take any class $x\in H_*s_qE$ with $d_1(x)=0$. We show that it is a permanent cycle, i.e. $d_i(x)=0$ for any $i \geq 1.$
	
	Consider the sequence of maps
	$$H_{*}f_qE\hookrightarrow H_{*}s_qE \xrightarrow{\delta_q} H_{*-1}f_{q+1}E\hookrightarrow H_{*-1}s_{q+1}E,$$
	where the first and last arrows are inclusion by assumption. Therefore $d_1(x)=0$ implies that $\delta_q(x)=0.$ Hence we have that the class $x$ lifts to $H_{*}f_qE$. When  $q=0,$ the class $x$ lifts to $H_{*}f_0E$ and is therefore a permanent cycle. 
	When $q\geq 1$, since the connecting homomorphism $H_{*+1}s_{q-1}E\to H_{*}f_qE$ is surjective, $x$ is the target of a $d_1$ differential. In particular, it is a permanent cycle.
\end{proof}
\begin{rem2}\label{Rmk:remark}
	From the proof of Lemma \ref{Lem:lemma}, we can see that when $H_*(f_qE\to f_{q-1}E)$ is zero for each $q\geq 1$, the elements in $H_*E$ are all detected by the classes in filtration zero. For higher filtrations $s_qE$, the classes either support a differential (i.e. have nonzero image under the connecting homomorphism), or are hit by a $d_1$ differential (i.e. lift to $f_qE$). 
\end{rem2}

With some extra assumptions, the converse of Lemma \ref{Lem:lemma} is true.

\begin{lem}
\label{Lem:converselemma}
	We use the same notation as in Lemma \ref{Lem:lemma}.
	The map $H_*(f_qE\to f_{q-1}E)$ is zero for each $q\geq 1$
if the tower satisfies the following conditions:
\begin{enumerate}
	\item the map $H_*(f_1E\to f_0E)$ is zero;
	\item the associated spectral sequence collapses at $E_2$-page;
	\item $\displaystyle\varprojlim_n f_nE=*. $
\end{enumerate}
\end{lem}

\begin{proof}
We prove the statement by induction. Assume the statement is true for $q\leq n+1.$ We prove that $H_*(f_{n+2}E\to f_{n+1}E)$ is zero, or equivalently, $H_*(j_{n+1}: f_{n+1}E\to s_{n+1}E)$ is injective.

We show this by contradiction. Suppose there is a nonzero class $0\neq x$ in $H_{*}f_{n+1}E$ such that $j_{n+1}(x)=0.$ The connecting homomorphism $\delta_n:H_{*+1}s_nE\to H_{*}f_{n+1}E$ is surjective by inductive hypotheses. We can choose a preimage $y\in \delta_n^{-1}(x).$ It follows that $d_1(y)=j_{n+1}\delta_n(y)=0$. By the collapsing condition, the class $y$ is a permanent cycle. Therefore, the class $x$ lifts to $\displaystyle\varprojlim_n H_*(f_nE)$. 
Since the spectral sequence collapses at a finite page, 
this inverse limit $\displaystyle\varprojlim_n H_*(f_nE)\cong H_*(\displaystyle\varprojlim_n f_nE)\cong 0$. This contradicts the assumption that $x$ is nonzero.
\end{proof}

We now discuss the sliceability of the $m$-th motivic Brown--Peterson spectrum $BPGL\langle m \rangle.$ For the sake of simplicity, we restrict to the case when $p=2$. The odd primary cases work similarly.

\begin{thm}
\label{Thm:BPGLhomology}
Over any base field $k$ of characteristic zero, we have that
	$$H_{**}BPGL\langle m\rangle\cong A//E(m)_{**} \cong \m^k_2[\overline{\tau_{m+1}},\dots,\overline{\xi_1},\dots]/(\overline\tau_i^2 - \tau\overline\xi_{i+1} - \rho\overline\tau_{i+1}, i\geq m+1).$$
	In particular,
	$$H_{**}H\z_p\cong A//E(0)_{**}\cong \m^k_2[\overline{\tau_{1}},\dots,\overline{\xi_1},\dots]/(\overline\tau_i^2 - \tau\overline\xi_{i+1} - \rho\overline\tau_{i+1}, i\geq 1).$$
\end{thm}

\begin{proof}
	The first isomorphism is \cite[Thm. 3.9]{Orm11}. The second isomorphism can be obtained by direct computation. For example, see \cite{Hil11, HKO17, Kylling2017, OO13} for the formulas of the $k$-motivic $A_{**}$, $E(m)$, and conjugation.
\end{proof}

Take the mod $2$ homology of the type $1$ effective slice tower of $BPGL\langle m\rangle$. By \cite[Cor. 4.7]{LT15}, the $q$-th slice is a wedge sum of suspensions of $H\z_p:$
$$s_q^1(BPGL\langle m\rangle)\simeq \Sigma^{2q,q}H\z_2\otimes\pi_{2q}BP\langle m\rangle.$$

We obtain a homological effective slice spectral sequence:
\begin{equation}\label{Eqn:HVSSS}
E_1^{***}=H_{**}(s^1_*BPGL\langle m\rangle)\cong H_{**}H\z_2[v_1,\dots,v_m]\implies H_{**}BPGL\langle m\rangle.
\end{equation}

For a fixed bidegree $(s,t)$, we have that $H_{s,t}(s^1_qBPGL\langle n \rangle)\simeq 0$ when $s-t<q.$ There are only finitely many slices that can contribute to a degree. On the other hand, recall that we have $\sc(BPGL\langle n \rangle)\simeq BPGL\langle m \rangle$ by \cite[Thm. 8.12]{Hoy15}. Therefore the spectral sequence strongly converges to $H_{**}BPGL\langle m \rangle$.

We will refer to this spectral sequence as the $H_*$-ESSS below. 
We have the following result about this spectral sequence:

\begin{thm}
\label{Thm:BPGLHVSS}
	Let $k$ be either $\q$, $\r$, or $\c$. The differentials in the spectral sequence \eqref{Eqn:HVSSS} are determined by $d_{2^i-1}(\bar\tau_i)=v_i$, $1\leq i\leq m$, and $\m^k_2$-linearity.	
\end{thm}

\begin{proof}
	We first treat the case $k=\r$. We prove the result by induction.

	When $m=1$, the spectral sequence \eqref{Eqn:HVSSS} takes the form
	$$H_{**}H\z_2[v_1] \Rightarrow H_{**}kgl,$$
where by Theorem \ref{Thm:BPGLhomology}, the abutment is
$$\m^k_2[\overline{\tau_{2}},\dots,\overline{\xi_1},\dots]/(\overline\tau_i^2 - \tau\overline\xi_{i+1} - \rho\overline\tau_{i+1}, i\geq 2),$$
and the $E_1$-page is 
$$\m^k_2[v_1][\overline{\tau_{1}},\dots,\overline{\xi_1},\dots]/(\overline\tau_i^2 - \tau\overline\xi_{i+1} - \rho\overline\tau_{i+1}, i\geq 1).$$
The degree $(s,q,u)$ of the elements and the differentials are as follows:
$$\lvert \bar\tau_i\rvert = (2^{i+1}-1, 0, 2^i-1), ~~ \lvert \bar\xi_i\rvert = (2^{i+1}-2, 0, 2^i-1), ~~\lvert v_i\rvert = (2^{i+1}-2, 2^i-1, 2^i-1), ~~i\geq 1;$$
$$\lvert \rho\rvert = (-1,0,-1), ~~\lvert \tau\rvert = (0,0,-1); ~~ \lvert d_r\rvert=(-1,r,0),~~ r\geq 1$$
By comparing degrees, the only possible target of $\bar\tau_{1}$ is $v_1.$ That determines all the differentials.

Now suppose that the statement holds for $m-1$. Consider the quotient map 
$$BPGL\langle m\rangle\to BPGL\langle m-1\rangle$$
and the induced map between spectral sequences. We have $d_{2^i-1}(\bar\tau_i)=v_i$ for $1\leq i\leq m-1$ by naturality, so the $E_{2^{m-1}}$-page is 
$$\m^k_2[v_m][\overline{\tau_m}, \dots, \overline{\xi_1}, \dots]/(\overline{\tau_i^2} - \tau\overline\xi_{i+1} - \rho\overline\tau_{i+1}; ~i\geq m).$$
By comparing degrees, the generator $\overline{\tau_m}$ supports a $d_{2^{m}-1}$ differential and hits $v_m.$ This completes the proof for $k=\r$.

We now consider the other base fields. The collections of bidegrees where $\m_2^\c$ or $\m_2^\q$ are nonzero are contained in the collection of bidegrees where $\m_2^\r$ is nonzero, c.f. \cite[Sec. 5]{OO13}. Since we only used degree arguments to deduce the differentials above, we obtain the desired result by the same proof. 
\end{proof}

\begin{thm}
\label{Thm:BPNotTypeOneSliceable}
	Let $m \geq 2$ and let $k \subseteq \c$ be a field which admits a complex embedding. The $m$-th motivic Brown--Peterson spectrum $BPGL\langle m \rangle \in \SH(k)$ is not algebraically $1$-sliceable.
\end{thm}

\begin{proof}

The $k=\c$ case follows from Theorem \ref{Thm:BPGLHVSS} and the contrapositive of Lemma \ref{Lem:lemma}.

More generally, let $k \hookrightarrow \c$ be a field which admits a complex embedding. Base change induces a map between the $H_*$-ESSS in $\SH(k)$ and the $H_*$-ESSS in $\SH(\c)$. The target does not collapse at $E_2$, so it follows from inspection of the map that the source cannot collapse at $E_2$. Indeed, base change sends the generators $\bar{\tau}_i$ and $v_i$ to the generators with the same name. Therefore we may apply Lemma \ref{Lem:lemma} to conclude that $BPGL \langle m \rangle \in \SH(k)$ is not algebraically sliceable.
\end{proof}

If we work with the type $m$ effective slice tower, then in the $H_*$-$m$-ESSS\footnote{The spectral sequence obtained by applying homology to the type $m$ effective slice tower.} for $BPGL \langle m \rangle$, the $q$-th filtration quotient in the $E_1$ term, $H_{**}s^m_q BPGL \langle m \rangle$, is an extension of the type $1$ filtration quotients $H_{**}s^1_i BPGL \langle m \rangle$ with $q(2^m-1) \leq i \leq (q+1)(2^m-1)-1.$ In other words, the homology of the type $m$ slices can be computed using the truncated $H_*$-ESSS. As a result, in the $H_*$-$m$-ESSS of $BPGL \langle m \rangle$, the longest differential has length $1$. 

Although $BPGL \langle m \rangle$, $m \geq 2$, is not algebraically $1$-sliceable, it \emph{is} algebraically $m$-sliceable. We have the following result.
	\begin{thm}
	\label{Thm:BPGLmSLiceable}
	Let $k$ be a field of characteristic zero. The $m$-th motivic Brown--Peterson spectrum $BPGL\langle m \rangle \in \SH(k)$ is algebraically $m$-sliceable. 
	\end{thm}

\begin{proof}
We first show the case $k=\q$.
We show the three assumptions in Lemma \ref{Lem:converselemma} are satisfied.

\begin{enumerate}
	\item 
	By Remark \ref{Rmk:remark}, we have that the map in the type $m$ tower
	$$BPGL\langle m \rangle \to s^m_0(BPGL\langle m \rangle)$$ 
induces an inclusion on homology. Therefore $H_{**}(f_1BPGL\langle m \rangle\to f_0BPGL\langle m \rangle)$ is zero.
	\item By the discussion above, the longest differential is of length $1$. As a result, the spectral sequence collapses at the $E_2$ page.
	\item Since $\sc(BPGL\langle m \rangle)\simeq BPGL\langle m \rangle$, the inverse limit of the effective slice tower of $BPGL\langle m\rangle$ is contractible. Equivalently, the inverse limit of the type $m$ effective slice tower is contractible.
\end{enumerate}

The result for $k=\q$ follows by applying Lemma \ref{Lem:converselemma}.

If $k$ is any field of characteristic zero, then there exists a field homomorphism $i: \q \to k$. We have
$$i^*(f_{q+1}^m BPGL\langle m \rangle \to f_q^m BPGL\langle m \rangle) \simeq f_{q+1}^mBPGL \langle m \rangle \to f_{q}^m BPGL \langle m \rangle$$
by naturality of the slice filtration, where the left-hand side is $i^*$ applied to part of the type $m$ slice tower over $\q$ and the right-hand side is part of the type $m$ slice tower over $k$. The left-hand side is zero, so the right-hand side is zero and thus $BPGL \langle m \rangle$ is algebraically $m$-sliceable over $k$. 
\end{proof}

\subsubsection{Very effective cover of hermitian K-theory}

We conclude by mentioning an example where the effective and very effective slice filtrations differ. 

\begin{rem2}\label{Rmk:kqsliceable}
Let $p=2$ and let $k$ be a perfect field of characteristic not two. A quadruple speed very effective slice filtration can be defined by setting $\bar{f}_q = \tilde{f}_{4q}$ and $\bar{s}_q = \cofib(\bar{f}_{q+1} \to \bar{f}_q)$. By \cite{Bac17}, the very effective cover of hermitian K-theory $kq$ \cite{ARO17} satisfies $\bar{f}_qkq \simeq \Sigma^{8q,4q}kq$. Observe that $\bar{f}_{q+1}kq \to \bar{f}_qkq$ may be identified with the Bott map $\beta: \Sigma^{8q+8,4q+4} kq \to \Sigma^{8q,4q} kq$ which induces the zero map in homology. It follows (c.f. the proof for $kgl$) that $kq$ is ``algebraically sliceable" with respect to the quadruple speed very effective slice filtration. Unfortunately, we cannot compute $H_{**}(\bar{s}_q kq)$ for any $q\geq 0$, so have been unsuccessful in attempts to apply Theorem \ref{Thm:comparisonSquare} to understanding the VSSS or mASS for $kq$. 

It is also worth noting that the very effective slice filtration of $kq$ does \emph{not} coincide with the effective slice filtration. It seems unlikely that $kq$ is algebraically sliceable with respect to a quadruple speed effective slice filtration. 
\end{rem2}

\subsection{Non-example}\label{SS:NonExamples}

We mention a non-example to indicate the limitations of our techniques. 

\begin{prop}\label{Prop:mglNot}
	The algebraic cobordism spectrum $MGL$ is not algebraically sliceable. 
\end{prop}

\begin{proof}
	The slices of $MGL$ are \cite[Theorem 8.5]{Hoy15}\cite[Thm. 4.7]{Spi10}
	\[
	s_qMGL \simeq \Sigma^{2q,q}H\pi_{2q}MU.
	\]
	In particular, we have that $s_0MGL \simeq H\z_p$, and that the map 
	\[
	MGL\to s_0MGL\simeq H\z_p
	\]
	coincides with the Thom class.
	
	Furthermore, by \cite[Theorem 6.5]{Hoy15}, we have an isomorphism of left $A_{**}$-comodules
	\[
	H_{**}MGL\simeq P_{**}\otimes \f_p[x_i\mid i\neq p^r-1],
	\]
	where 
	\[
	P_{**} = \m_p^k[\xi_1, \xi_2, \ldots]\subseteq A_{**}
	\]
	is the even subalgebra of the motivic dual Steenrod algebra. The map 
	
	\[
	H_{**}MGL\to H_{**}H\z_p\cong A// E(0)_{**}
	\]
	is the inclusion of $P_{**}$ and 0 on the other factor. Thus the fiber sequence
	\[
	f_1 MGL\to MGL\to s_0MGL
	\]
	induces a long exact sequence in mod $p$ motivic homology which is not short exact. This shows that $MGL$ is not algebraically sliceable. 
\end{proof}

\subsection{The sphere spectrum}

We conclude our discussion of examples and non-examples by discussing the sphere spectrum $S^{0,0} \in \SH(k)$.

First, we observe that as in the case of the very effective cover of hermitian K-theory $kq$ (see Remark \ref{Rmk:kqsliceable}), the effective and very effective slice filtrations of the sphere spectrum are different. Unlike $kq$, however, we do not know if the sphere spectrum is algebraically sliceable with respect to any variant of either slice filtration. 

In any case, it is interesting to suppose that $S^{0,0}$ is algebraically sliceable and to speculate on the behavior of its algebraic slice spectral sequence. Such a spectral sequence would have the form
\begin{equation}
E_1^{q,s,t,u} = \bigoplus_{i \geq 0} \Ext_{A^\vee}^{s,t,u}(H_{**}s_qS^{0,0}) \Rightarrow \Ext_{A^\vee}^{s+q,t+q,u}(\m_p^k).
\end{equation}
By \cite[Thm. 2.12]{RSO19}, the effective slices of the sphere spectrum are intimately connected to the $E_2$-page of the classical Adams--Novikov spectral sequence. That is, there is an equivalence of motivic spectra
\begin{equation}
s_q((S^{0,0})_p^\wedge) \simeq \bigvee_{i \geq 0} \Sigma^{2q-i,q} H(\Ext_{BP_*BP}^{i,2q}(BP_*,BP_*)),
\end{equation}
where $BP$ is the classical $p$-primary Brown--Peterson spectrum. Therefore the algebraic slice spectral sequence for the sphere spectrum, if it existed, would begin with the $\Ext$-groups of the entries in the $E_2$-term of the classical Adams--Novikov spectral sequence and end with the $E_2$-term of the $k$-motivic Adams spectral sequence.

\section{Comparison of differentials}\label{Sec:Differentials}

We now apply Theorem \ref{Thm:comparisonSquare} to compare differentials in the effective slice spectral sequence, the motivic Adams spectral sequence, and the $\rho$-Bockstein spectral sequence. 

\subsection{The $\rho$-Bockstein spectral sequence}\label{SS:RhoBSS}

When $k=\c$, the groups $\Ext_{A^\vee}^{***}(\m_p^\c,H_{**}(X))$ are roughly as complicated as their classical counterparts $\Ext_{A^\vee}^{**}(\f_p,H_*(\Re(X)))$. When $k$ is an arbitrary base field, the groups $\Ext_{A^\vee}^{***}(\m_p^k,H_{**}(X))$ contain more complicated arithmetic data coming from $\m_p^k$:

\begin{thm}[Voevodsky]\cite[Thm. 2.7]{IO18}
Suppose that $p$ and $\chara(k)$ are coprime, and that $k$ contains a primitive $p$-th root of unity. Then there is an isomorphism
$$\m_p^k \cong K^M_*(k)/p[\tau]$$
where $K_n^M(k)$ has degree $(-n,-n)$ and $|\tau|$ has degree $(0,-1)$.\footnote{Recall that we have defined $\m_p^k$ as the mod $p$ motivic \emph{homology} of a point, so the bidegree of elements in $K_n^M(k)$, as well as $\tau$, are negative instead of positive.}
\end{thm}

\begin{exm}\cite[Exs. 2.1, 2.2, 2.6]{IO18}\cite[Prop. 2.4.2]{Sta16}
We record $\m_p^k$ for several cases of $k$ considered later:
\begin{enumerate}
\label{exm:coefficient_ring}
\item $\m_p^\c \cong \f_p[\tau]$ with $|\tau| = (0,-1)$. 

\item $$\m_p^\r \cong \begin{cases}
\f_2[\tau,\rho] \quad & \text{ if } p=2, \\
\f_p[\theta] \quad & \text{ if } p \neq 2,
\end{cases}$$
where $|\rho| = (-1,-1)$, $|\tau| = (0,-1)$, and $|\theta| = (0,-2)$. 

\item $\m_2^{\f_q} \cong \f_2[\tau,u]/u^2$ with $|u| = (-1,-1)$ and $|\tau| = (0,-1)$. We note that the $A^\vee$-coaction on $u$ depends on the equivalence class of $q$ modulo $4$, and if $q \equiv 3 \mod 4$, then $u = \rho$. 

\end{enumerate}
\end{exm}

Recall that $\Ext_{A^\vee}^{***}(\m_p^k,H_{**}(X))$ may be calculated using the cobar complex $C_\bullet(\m_p^k, A^\vee, H_{**}(X))$. When $k=\r$ and $p=2$, filtering the cobar complex by powers of $\rho$ gives rise to the \emph{$\rho$-Bockstein spectral sequence} \cite{Hil11}
\begin{equation}\label{Eqn:RhoBocksteinSS}
E_1^{****} = \Ext_{A^\vee}^{***}(\m_2^\c,H_{**}(X))[\rho] \Rightarrow \Ext_{A^\vee}^{***}(\m_2^\r,H_{**}(X))
\end{equation}
with differentials $d_r: E_r^{q,s,t,u} \to E_r^{q+r,s+1,t,u}$ calculated using the coaction of $A^\vee$ on $\m_2^\r$. The $\rho$-Bockstein spectral sequence can be built over more general base fields (e.g. $k= \f_q$ \cite{WO17}) using analogous constructions. 

Hill applied the $\rho$-Bockstein spectral sequence to compute $\Ext_{E(n)}^{***}(\m_2^\r,\m_2^\r)$ for all $n \geq 0$ \cite[Thm. 3.1]{Hil11} and $\Ext_{A(1)}^{***}(\m_2^\r,\m_2^\r)$ \cite[Fig. 6]{Hil11}. In both cases, the differentials were computed using an explicit analysis of the cobar complex and technical Massey product arguments. The groups $\Ext_{A(1)}^{***}(\m_2^\r,\m_2^\r)$ were recomputed by Guillou--Hill--Isaksen--Ravenel \cite[Sec. 6]{GHIR19} using a result comparing the $\rho$-inverted $\r$-motivic and classical $\Ext$ groups \cite[Prop. 4.1]{GHIR19} to force all of the necessary differentials. 

	As mentioned in the proof of Theorem \ref{Thm:knAlg}, $H^{**}(k(n)) \cong A//E(Q_n)$. Its mASS $E_2$-page is the the trigraded group $\Ext_{E(Q_n)}^{***}(\m_2, \m_2).$ Restricting to $\c$ and $\r$, we have the following result on the $\rho$-Bockstein spectral sequence associated to $k(n)$; compare with \cite[Thm. 6.3]{Hea19} for the case $n=1$. 

\begin{prop}\label{Prop:RhoBocksteinEQn}
For all $n \geq 0$, the $E_1$-page of the $\rho$-Bockstein spectral sequence converging to $\Ext_{E(Q_n)}^{***}(\m_2^\r, \m_2^\r)$ is given by
$$E_1 = \Ext^{***}_{E(Q_n)}(\m_2^\c,\m_2^\c)[\rho] \cong \f_2[\rho, \tau, v_n].$$
The nontrivial differentials are generated under $\rho$- and $v_n$-linearity by the differentials
$$d_{2^{n+1}-1}(\tau^{2^n}) = \rho^{2^{n+1}-1}v_n.$$
\end{prop}

\begin{proof}
The result about the $E_1$ page follows from direct computation.

The result about the differentials can be obtained by a similar computation as the $\rho$-Bockstein spectral sequence computation for $E(n)$ (\cite[Thm. 3.2]{Hil11}). Here we analyze it by comparison. The $\rho$-Bockstein spectral sequence for $E(n)$ has a differential hitting $\rho^{2^{n+1}-1}v_n$ and its $v_n$- and $\rho$-power multiples. Therefore, in our case, the elements with the same names must be killed. By degree reasons, the only possibility is $d_{2^{n+1}-1}(\tau^{2^n}) = \rho^{2^{n+1}-1}v_n$ and differentials generated by it under $\rho$ and $v_n$-linearity.
	
	\end{proof}

\begin{rem2}
Although the $\rho$-Bockstein spectral sequence for $\Ext_A^{***}(\m_2^\r,\m_2^\r)$ is multiplicative \cite[Prop. 2.3]{Hil11}, the $\rho$-Bockstein spectral sequence for $\Ext_{E(Q_n)}^{***}(\m_2^\r,\m_2^\r)$ is not. Thus no issues arise from the fact that $\tau$ does not support a differential while $\tau^{2^n}$ does. This is analogous to the topological situation, where the ESSS for the sphere spectrum is multiplicative \cite[Prop. 2.24]{RSO19}, but the ESSS for $k(n)$ cannot be multiplicative: the differentials in the $n=1$ case of \cite[Thm. 6.3]{Hea19} are incompatible with a Leibniz rule.
\end{rem2}

\subsection{Review of Adams differentials}

In this section we record some facts about differentials in the motivic Adams spectral sequence. For the remainder of this section, we focus on the case $p=2$ since $\rho=0$ when $p$ is odd. 

\begin{lem}\label{Lem:mASSCollapse}
If $k$ is an algebraically closed field or $k=\r$, then the motivic Adams spectral sequences converging to $\pi_{**}(k(n))$ and $\pi_{**}(BPGL\langle n \rangle)$, $n \geq 0$, collapse at $E_2$.
\end{lem}

\begin{proof}
Recall $H^{**}k(n) \cong A//E(Q_n)$ and $H^{**}BPGL\langle n \rangle \cong A//E(n)$. If $k$ is algebraically closed, then 
$${^{\mASS}E}_2(k(n)) = \Ext_{E(Q_n)}^{***}(\m_2^k,\m_2^k) \cong \Ext_{E(Q_n)}^{***}(\m_2^\c,\m_2^\c) \cong \f_2[\tau,v_n].$$
$${^{\mASS}E}_2(BPGL \langle n \rangle) = \Ext_{E(n)}^{***}(\m_2^k,\m_2^k) \cong \Ext_{E(n)}^{***}(\m_2^\c,\m_2^\c) \cong \f_2[\tau,v_1,\ldots,v_n].$$
Both spectral sequences collapse for tridegree reasons. 

When $k=\r$, this follows from \cite[Thm. 5.3]{Hil11} for $BPGL \langle n \rangle$. Similarly, the spectral sequence collapses at $E_2$ for $k(n)$ for tridegree reasons. 
\end{proof}

On the other hand, there are possible differentials when $k=\f_q$.

\begin{lem}
\label{lem_mASSE_2}
When $q \equiv 1 \mod 4$, the $E_2$-page of the $\f_q$-motivic Adams spectral sequence for $kgl$ is given by
$$E_2 = \Ext_{E(1)}^{***}(\m_2^{\f_q},\m_2^{\f_q}) \cong \f_2[\tau, u, v_0,v_1]/u^2.$$
When $q \equiv 3 \mod 4$, the $E_2$-page is given by
$$E_2 = \Ext_{E(1)}^{***}(\m_2^{\f_q},\m_2^{\f_q}) \cong \f_2[\tau^2,\rho, [\rho\tau], v_0,v_1]/(\rho^2, \rho[\rho\tau],[\rho\tau]^2,\rho v_0).$$
\end{lem}

\begin{proof}
When $q \equiv 1 \mod 4$, the class $u \in \m_2^{\f_q}$ satisfies $\eta_R(u) = \eta_L(u)$. Therefore it is a permanent cycle in the cobar complex and we see that $\Ext^{\f_q} \cong \Ext^{\c}[u]/u^2$. 

When $q \equiv 3 \mod 4$, the result follows from the $\rho$-Bockstein spectral sequence\footnote{In this case, just a long exact sequence.}. In particular, the $\rho$-Bockstein $d_1$-differentials are generated under the Leibniz rule by $d_1(\tau) = \rho v_0$, and the spectral sequence collapses at $E_2$. 
\end{proof}

Kylling showed that the $\f_q$-motivic Adams spectral sequences for $H\z_2$, $kgl$, and $kq$ do not collapse \cite[Sec. 4.2]{Kyl15}. 

\begin{thm}\label{Thm:mASSBPGL}
Let $\nu_2$ denote $2$-adic valuation. When $q \equiv 1 \mod 4$, the nontrivial differentials in the motivic Adams spectral sequence converging to $\pi_{**}^{\f_q}(BPGL\langle m \rangle)$ are generated under the Leibniz rule by
$$d_{\nu_2(q-1)+s}(\tau^{2^s}) = u \tau^{2^s-1} h_0^{\nu_2(q-1)+s}, \quad s \geq 0.$$
When $q \equiv 3 \mod 4$, they are generated by
$$d_{\nu_2(q^2-1)+s-1}(\tau^{2^s}) = \rho \tau^{2^s-1} h_0^{\nu_2(q^2-1)+s}, \quad s \geq 1.$$
\end{thm}

\begin{proof}
For all $n\geq 1$, there is a map
$$BPGL \langle m \rangle \to BPGL\langle 0 \rangle \simeq H\z_2.$$
The claimed differentials then follow by naturality of the motivic Adams spectral sequence from the differentials for $H\z_2$ which were calculated in \cite[Lems. 4.2.1, 4.2.2]{Kyl15}. 
\end{proof}

\subsection{Relating $\rho$-Bockstein and slice differentials over $\Spec(\r)$}\label{SS:ComparisonR}

Throughout this section, we fix $k=\r$. The square \eqref{Eqn:comparisonSquare} allows us to describe a precise relationship between differentials in the $\rho$-BSS and $n$-ESSS for $kgl$ and $k(n)$. We do so by first relating the $\rho$-BSS differentials with $n$-aESSS differentials, and then we will relate $n$-aESSS differentials to $n$-ESSS differentials.

\begin{prop}\label{Prop:rhoalg}
There are 1-to-1 correspondences between the following differentials:
\begin{enumerate}
\item $\rho$-BSS $d_3$-differentials and aESSS $d_1$-differentials for $kgl$.
\item $\rho$-BSS $d_{2^{n+1}-1}$-differentials and $n$-aESSS $d_1$-differentials for $k(n)$.
\end{enumerate}
Moreover, the aESSS and $n$-aESSS differentials are forced by knowledge of the corresponding $\rho$-BSS differentials. 
\end{prop}

\begin{proof}
Throughout this proof, we will write $\Ext_A^{***}(E)$ in place of $\Ext_A^{***}(H_{**}E)$ to avoid clutter. 

Let $E$ denote $kgl$ or $k(n)$, $n \geq 1$, let $i^* : \SH(\r) \to \SH(\c)$ denote base change along the inclusion $\r \to \c$, and let $\bar{s}_q$ denote $s_q$ if $E = kgl$ or $s_q^n$ if $E = k(n)$. 

We will prove the proposition by considering the square of spectral sequences
\begin{equation}\label{Eqn:RhoAlg}
\begin{tikzcd}
\bigoplus_{q \geq 0} \Ext_{A^\c}^{***}(\bar{s}_qi^*E)[\rho] \ar[r, Rightarrow, "\rho-BSS"] \ar[d, Rightarrow] & \bigoplus_{q \geq 0} \Ext_{A^\r}^{***}(\bar{s}_qE) \ar[d, Rightarrow] \\
\Ext_{A^\c}^{***}(i^*E)[\rho] \ar[r,Rightarrow,"\rho-BSS"] & \Ext_{A^\r}^{***}(E).
\end{tikzcd}
\end{equation}
where the rows are $\rho$-Bockstein spectral sequences, the left column is the direct sum over powers of $\rho$ of the aESSS or $n$-aESSS, and the right column is the aESSS or $n$-aESSS.\footnote{The top left corner is the $E_1$-page of the $\rho$-BSS because $\bar{s}_qi^*E \simeq i^*\bar{s}_qE$, which follows from the observation that $\bar{s}_qE$ is Eilenberg--MacLane for all $q \in \z$ and base change preserves Eilenberg--MacLane spectra.}

Let $E = kgl$. We can identify each corner of \eqref{Eqn:RhoAlg}. We have
$$\bigoplus_{q \geq 0} \Ext_{A^\c}^{***}(i^*s_q kgl)[\rho] \cong \bigoplus_{q \geq 0}\Ext_{A^\c}^{***}(\Sigma^{2q,q}i^*H\z_2)[\rho] \cong \bigoplus_{q \geq 0} \Sigma^{2q,q}\f_2[\tau,\rho,v_0],$$
$$\bigoplus_{q \geq 0} \Ext_{A^\r}^{***}(s_q kgl) \cong \bigoplus_{q \geq 0} \Ext_{A^\r}^{***}(\Sigma^{2q,q}H\z_2) \cong \bigoplus_{q \geq 0} \Sigma^{2q,q}\f_2[\tau^2,\rho,v_0]/(\rho v_0),$$
$$\Ext_{A^\c}^{***}(i^*kgl)[\rho] \cong \f_2[\tau,\rho,v_0,v_1],$$
$$\Ext_{A^\r}^{***}(kgl) \cong \f_2[\tau^4,\rho,v_0,v_1]/(\rho v_0, \rho^3 v_1)$$
where $|\tau| = (0,0,-1)$, $|\rho| = (0,-1,-1)$, $|v_0| = (1,0,0)$, and $|v_1| = (1,3,1)$. 

The left column collapses for quad-degree reasons: the $E_1$-page of the aESSS for $i^*kgl$ is concentrated in quad-degrees with $t-s \equiv 0 \mod 2$, but $d_r$ decreases $t-s$ by $1$ for all $r \geq 1$. The differentials in the top row are generated by 
$$d_1(\tau) = \rho v_0$$
and the differentials in the bottom row are generated by
$$d_1(\tau) = \rho v_0, \quad d_3(\tau^2) = \rho^3 v_1.$$
Since the upper right and lower left composites must arrive at the same answer, we must have
$$d_1(\tau^2) = \rho^3 v_1$$
in the right column, where by abuse of notation $v_1$ denotes the generator of $\Ext_{A^\r}^{***}(s_1kgl)$. 

Now fix $n\geq1$ and let $E = k(n)$. To declutter the notation, let $r = 2^n-1$. For the corners of \eqref{Eqn:RhoAlg}, we have
$$\bigoplus_{q \geq 0} \Ext_{A^\c}^{***}(i^*s_q^nk(n))[\rho] \cong \bigoplus_{q \geq 0} \Ext_{A^\c}^{***}(\Sigma^{2qr,qr}i^*H\f_2) \cong \bigoplus_{q \geq 0} \Sigma^{2qr,qr} \f_2[\tau,\rho],$$
$$\bigoplus_{q \geq 0} \Ext_{A^\r}^{***}(s_q^nk(n)) \cong \bigoplus_{q \geq 0} \Ext_{A^\r}^{***}(\Sigma^{2qr,qr}H\f_2) \cong \bigoplus_{q \geq 0} \Sigma^{2qr,qr} \f_2[\tau,\rho],$$
$$\Ext^{***}_{A^\c}(k(n))[\rho] \cong \f_2[\tau,\rho,v_n],$$
$$\Ext^{***}_{A^\r}(k(n)) \cong \f_2[\tau,\rho,t_{n+1},v_n]/(\tau^{2^n},\rho^{2^{n+1}-1}v_n)$$
where $|\tau| = (0,0,-1)$, $|\rho| = (0,-1,-1)$, $|v_n| = (1,2^{n+1}-1,2^n-1)$ and $|t_n| = (0,0,-2^{n+1})$. Here, $t_n$ is represented by $\tau^{2^{n+1}} $ on the $E_1$-page of the $\rho$-BSS, c.f. \cite[Thm. 6.3]{Hea19}. 

The left column and top row both collapse for quad-degree reasons, and the differentials in the bottom row are determined by
$$d_{2^{n+1}-1}(\tau^{2^n}) = \rho^{2^{n+1}-1}v_n.$$
We conclude by commutativity of \eqref{Eqn:RhoAlg} that we must have
$$d_1(\tau^{2^n}) = \rho^{2^{n+1}-1}v_n$$
in the $n$-aESSS, where by abuse of notation $v_n$ denotes the generator of $\Ext_{A^\r}^{***}(s_1^nk(n))$. 
\end{proof}

\begin{lem}\label{Lem:algtoSSS}
There are 1-to-1 correspondences between the following differentials:
\begin{enumerate}
\item aESSS $d_1$-differentials and ESSS $d_1$-differentials for $kgl$. 
\item $n$-aESSS $d_1$-differentials and $n$-ESSS $d_1$-differentials for $k(n)$. 
\end{enumerate}
Moreover, the ESSS and $n$-ESSS differentials are forced by knowledge of the corresponding aESSS and $n$-aESSS differentials.
\end{lem}

\begin{proof}
We have shown $kgl$ is algebraically sliceable and $k(n)$ is algebraically $n$-sliceable, so in the notation of the previous proof, we have squares of spectral sequences
\begin{equation}
\begin{tikzcd}
& \bigoplus_{q \geq 0} \Ext^{***}_{A}(H_{**}\bar{s}_q E) \arrow[dl,Rightarrow,"aESSS"'] \arrow[dr,Rightarrow,"\bigoplus_{q \geq 0} mASS"] & \\
\Ext^{***}_A(H_{**}E) \arrow[dr,Rightarrow,"mASS"'] & & \bigoplus_{q \geq 0} \pi_{**}\bar{s}_q E \arrow[dl,Rightarrow,"ESSS"] \\
& \pi_{**}E. 
\end{tikzcd}
\end{equation}
for $E = kgl$ and $E = k(n)$. 

The mASS's for $kgl$, $k(n)$, $s_qkgl \simeq \Sigma^{2q,q}H\z_2$, and $s_q^n k(n) \simeq \Sigma^{2qr,qr}H\f_2$ all collapse. The only nontrivial differentials in the aESSS for $kgl$ and the $n$-aESSS for $k(n)$ have length one, so for quad-degree reasons, we must have identical differentials in the ESSS for $kgl$ and $n$-ESSS for $k(n)$. 
\end{proof}

Putting together Proposition \ref{Prop:rhoalg} and Lemma \ref{Lem:algtoSSS}, we have proven the following theorem. 

\begin{thm}\label{Thm:RCorrespondence}
There are 1-to-1 correspondences between the following differentials:
\begin{enumerate}
\item $\rho$-BSS $d_3$-differentials and ESSS $d_1$-differentials for $kgl$.
\item $\rho$-BSS $d_{2^{n+1}-1}$-differentials and $n$-ESSS $d_1$-differentials for $k(n)$. 
\end{enumerate}
Moreover, the ESSS and $n$-ESSS differentials are completely determined by knowledge of the corresponding $\rho$-BSS differentials.
\end{thm}

\begin{rem2}\label{Rmk:RCorrespondence}
As the $n$-ESSS is obtained by increasing the speed of the effective slice filtration, a $d_1$-differential in the $n$-ESSS corresponds to a $d_{2^{n}-1}$-differential in the ESSS. Therefore part (2) of Theorem \ref{Thm:RCorrespondence} may be restated as a 1-to-1 correspondence between $\rho$-BSS $d_{2^{n+1}-1}$-differentials and $d_{2^n-1}$-differentials in the ESSS. 
\end{rem2}

Theorem \ref{Thm:RCorrespondence} can be used to prove a correspondence between more complicated differentials. The following theorem is closely related to Yagita's analysis of the effective slice spectral sequence for $BPGL$ in \cite[Sec. 4]{Yag05}. 

\begin{thm}\label{Thm:BPGLCorrespondence}
Let $m \geq 1$. For each $1 \leq i \leq m$, there is a 1-to-1 correspondence between the $\rho$-BSS $d_{2^{i+1}-1}$-differentials and ESSS $d_{2^i-1}$-differentials for $BPGL\langle m \rangle$. 
\end{thm}

\begin{proof}
We proceed by induction on $m$. At $p=2$, there is an equivalence $kgl \simeq BPGL \langle 1 \rangle$, so the base case $m=1$ is handled by Theorem \ref{Thm:RCorrespondence}. 

Suppose now that the result holds for all $k<m$ and we wish to show the theorem holds for $BPGL \langle m \rangle$. We begin by using the induction hypothesis to show that the theorem holds for $i<m$ (but not necessarily $i=m$). 

Observe that the $d_{2^{i+1}-1}$-differentials in the $\rho$-BSS's for $BPGL\langle m-1 \rangle$ and $BPGL \langle m \rangle$ are determined by the same differential
\begin{equation}
d_{2^{i+1}-1}(\tau^{2^i}) = \rho^{2^{i+1}-1}v_i
\end{equation}
for $1 \leq i \leq m-1$ by \cite[Thm. 3.2]{Hil11}. This implies that the $d_{2^{i+1}-1}$-differentials for $BPGL \langle m \rangle$ are precisely the $d_{2^{i+1}-1}$-differentials for $BPGL \langle m-1 \rangle$ extended $v_m$-linearly. 

On the other hand, naturality of the ESSS allows us to completely determine the $d_{2^i-1}$-differentials in  the ESSS for $BPGL \langle m \rangle$. Indeed, the quotient map $BPGL \langle m \rangle \to BPGL \langle m-1 \rangle$ implies that the $d_{2^i-1}$-differentials between slices indexed on monomials not divisible by $v_m$ coincide for $BPGL\langle m -1 \rangle$ and $BPGL \langle m \rangle$, and the quotient map $BPGL \langle m \rangle \to k(n)$ implies that $v_m$ is a $d_{2^i-1}$-cycle for $i<m$ by Theorem \ref{Thm:RCorrespondence}. Applying the Leibniz rule for the ESSS \cite[Prop. 2.24]{RSO19} proves that the $d_{2^i-1}$-differentials in the ESSS for $BPGL \langle m \rangle$ are precisely the $d_{2^i-1}$-differentials in the ESS for $BPGL \langle m-1 \rangle$ extended $v_m$-linearly. Therefore we have proven a 1-to-1 correspondence between $\rho$-BSS and ESSS differentials for $i<m$. 

We now consider the case $i=m$. Theorem \ref{Thm:BPGLmSLiceable} shows that $BPGL \langle m \rangle$ algebraically $m$-sliceable, so we may consider the comparison square \eqref{Eqn:comparisonSquare} relating its $m$-aESSS and $m$-ESSS. The mASS's for $BPGL \langle m \rangle$ and its type $m$ slices (extensions of $(2j,j)$-suspensions of $H\z_2$) all collapse, so we obtain a 1-to-1 correspondence between $d_1$-differentials in the $m$-aESSS and $m$-ESSS. 

We are therefore reduced to relating the $d_{2^{m+1}-1}$-differentials in the $\rho$-BSS and the $d_1$-differentials in the $m$-aESSS, but this follows from essentially the same arguments as Proposition \ref{Prop:rhoalg}. Indeed, consider the $BPGL \langle m \rangle$-analog of \eqref{Eqn:RhoAlg}:
\begin{equation}\label{Eqn:RhoAlgBPGL}
\begin{tikzcd}
\bigoplus_{q \geq 0} \Ext_{A^\c}^{***}({s}_q^m i^* BPGL \langle m \rangle)[\rho] \ar[r, Rightarrow, "\rho-BSS"] \ar[d, Rightarrow] & \bigoplus_{q \geq 0} \Ext_{A^\r}^{***}(s_q^m BPGL \langle m \rangle) \ar[d, Rightarrow] \\
\Ext_{A^\c}^{***}(i^* BPGL \langle m \rangle)[\rho] \ar[r,Rightarrow,"\rho-BSS"] & \Ext_{A^\r}^{***}(BPGL \langle m \rangle).
\end{tikzcd}
\end{equation}
The left column collapses for degree reasons and the differentials in the bottom row were determined in \cite[Thm. 3.2]{Hil11}. The differentials in the top row coincide with the differentials in the bottom row, with two types of exceptions:
\begin{enumerate}
\item All of the $d_{2^{m+1}-1}$-differentials, i.e. the longest nontrivial differentials, in the bottom row cannot occur in the top row since their target in the bottom is divisible by $v_m$ (which is zero in the top). We refer to these as \emph{long differentials}. 
\item The $d_r$-differentials for $r<2^{m+1}-1$ in the bottom row of the form $d_r(x)=y$ where $x$ lies in stem $s$ with $k(2^{m+1}-2) \leq s < (k+1)(2^{m+1}-2)$ and $y$ lies in stem $s'$ with $s' \geq (k+1)(2^{m+1}-2)$ cannot support differentials in the top row for degree reasons. We refer to these as \emph{fringe differentials} since they involve elements from adjacent type $m$ slices.
\end{enumerate}
A discussion fringe differentials for the case $m=2$ appears in Example \ref{Exm:Fringe}. 

The long and fringe differentials in the bottom row which do not occur in the top row both are accounted for by $d_1$-differentials in the $m$-aESSS. The $d_1$-differentials in the $m$-aESSS forced by the fringe differentials in the $\rho$-BSS give rise to $d_r$-differentials in the ESSS with $r< 2^{m}-1$, which are completely understood by the induction hypothesis. On the other hand, the $d_1$-differentials in the $m$-aESSS forced by the long differentials in the $\rho$-BSS correspond to $d_{2^m-1}$-differentials in the ESSS; this is precisely the correspondence we wished to show.

\end{proof}

\begin{exm}\label{Exm:Fringe}
	We analyze the fringe $d_r$-differentials in the $\rho$-BSS for $BPGL\langle 2\rangle$. In this case, we have the cofiber sequence
	\[
		f_3 BPGL\langle 2\rangle \to BPGL\langle 2\rangle \to s_0^2BPGL\langle 2\rangle. 
	\]
	Observe that 
	\[
		s_0^2BPGL\langle 2\rangle \simeq s_0^2BPGL\langle 1\rangle \simeq BPGL\langle 1\rangle/v_1^3. 
	\]
	Furthermore, the associated graded motivic spectrum $s_*^2BPGL\langle 2\rangle$ is given by 
	\[
		s_*^2BPGL\langle 2\rangle \simeq s_0^2 BPGL\langle 2\rangle [v_1^3, v_2].
	\]
	In particular, to understand the top $\rho$-BSS of \eqref{Eqn:RhoAlgBPGL} it is enough to understand the $\rho$-BSS for $s_0^2 BPGL\langle 2\rangle$. This can be done using the naturality of the $\rho$-BSS which produces the diagram
	\begin{equation}
		\begin{tikzcd}
			\Ext_{A^\c}^{***}(BPGL\langle 2\rangle)[\rho]\arrow[r,Rightarrow] \arrow[d] & \Ext_{A^\r}(BPGL\langle 2\rangle )\arrow[d] \\
			\Ext_{A^\c}^{***}(s_0^2BPGL\langle 2\rangle)[\rho]\arrow[r,Rightarrow] & \Ext_{A^\r}^{***}(s^2_0 BPGL\langle 2\rangle).
		\end{tikzcd}
	\end{equation}
	On $E_1$-pages, the left-hand map is the projection map 
	\[
		\f_2[\tau, \rho, v_0, v_1,v_2] \to \f_2[\tau, \rho, v_0, v_1,v_2]/(v_1^3, v_2). 
	\]
	This allows us to import differentials from the $\rho$-BSS for $BPGL\langle 2\rangle$. 

In the top and bottom rows, the $d_1$-differentials are determined by the differential $d_1(\tau)=\rho v_0$, and the morphism of spectral sequences is the obvious map on $E_2$-pages. There are no $d_2$-differentials so $E_2 = E_3$. 

We encounter a fringe differential on the $E_3$-page. In the top, there is a $d_3$-differential 
	\[
		d_3(\tau^2) = \rho^3v_1
	\]
	which implies a differential of the same form in the bottom. However, in the top there is a differential $d_3(\tau^2 v_1^2) = \rho^3v_1^3$ which cannot occur in the bottom since $v_1^3=0$. This is precisely the phenomena described by \emph{fringe differentials} in the previous proof. 

After running $d_3$-differentials, both rows collapse. Thus, in $\Ext_{A^\r}(s^2_0 BPGL\langle 2\rangle)$, there is an extra class which we denote by $[\tau^2 v_1^2]$ to remind us that it is indecomposable in $\Ext_{A^\r}$. 
	Since there are $\rho$-BSS differentials $d_3(\tau^2 v_1^2) = \rho^3v_1^3$ in the $\rho$-BSS for $BPGL\langle 2\rangle$, it follows that we must have differentials of the form 
	\[
		d_1(\rho^\ell[\tau^2v_1^2](v_1^3)^jv_2^k) = \rho^{\ell+3}(v_1^3)^{j+1}v_2^k
	\]
	in the $m$-aESS, and hence in the $m$-ESSS. These correspond to the $d_1$-differentials in the ESSS which are already understood by reduction to $BPGL \langle 1 \rangle \simeq kgl$. 
\end{exm}

\begin{rem2}
Naturality of the ESSS can be used to prove an analogous correspondence between $\rho$-BSS and ESSS differentials for $BPGL$. 
\end{rem2}

\begin{rem2}
Although we have only considered the case $k=\r$ here, it would be interesting to consider more general base fields of characteristic zero. For example, the mASS for $BPGL \langle n \rangle$ is completely understood when $k=\q$ by work of Ormsby and {\O}stv{\ae}r \cite[Thm. 5.8]{OO13}. If one extends Theorem \ref{Thm:BPGLmSLiceable} to $k=\q$, then it may be possible to understand the ESSS for $BPGL \langle n \rangle$ over $\Spec(\q)$ using Theorem \ref{Thm:comparisonSquare}. 
\end{rem2}

\subsection{Comparison over $\Spec(\f_q)$}\label{SS:CompareFq}

We now turn to the case when $k=\f_q$. 

Unlike the $\r$-motivic case, the comparison is trivial for $k(n)$, $n \geq 1$. Since $\m_2^{\f_p}$ is zero in bidegrees $(s,u)$ with $s < -1$, the $\rho$-BSS, mASS, and $n$-ESSS all collapse for degree reasons. 

However, the analysis is still interesting for $kgl$. Recall from Example \ref{exm:coefficient_ring}(3) that the $A^\vee$-coaction depends on the equivalence class of $q$ modulo $4$. We discuss the case when $q\equiv 1 \mod 4$, and the other case when $q\equiv 3\mod 4$ is similar. 

\subsubsection{$q\equiv 1\mod 4$}\hfill

By Lemma \ref{lem_mASSE_2}, we observe that in the motivic Adams spectral sequence of $kgl$ and $H\z_2,$ the $E_2$-pages have the relation that
$$\Ext^{***}_A(kgl)\cong \Ext^{***}_A(H\z_2)[v_1].$$

Since $\Ext^{***}_A(H\z_2)[v_1]$ is exactly the $E_1$-page of the algebraic effective slice spectral sequence for $kgl$, we conclude that the algebraic effective slice spectral sequence collapses.

By Theorem \ref{Thm:mASSBPGL}, the nontrivial differentials in both motivic Adams spectral sequences for $kgl$ and $H\z_2$ are determined under Leibniz rule by 
$$d_{\nu_2(q-1)+s} (\tau^{2^s})=u \tau^{2^s-1}h_0^{\nu_2(q-1)+s}, s\geq 0.$$ 

Therefore, using the motivic comparison square \eqref{Eqn:comparisonSquare}, we conclude that the effective slice spectral sequence for $kgl$ collapses.

\subsubsection{$q\equiv 3\mod 4$}\hfill

When $q\equiv 3\mod 4$, we also get that the algebraic effective slice spectral sequence and the effective slice spectral sequence for $kgl$ both collapse.

\begin{rem2}
	The work of Ormsby \cite{Orm11} shows that the situation over local fields is very similar. The above analysis can be carried over local fields to show that the effective slice spectral sequence for $kgl$, as well as $BPGL\langle m\rangle$ for all $m\geq 0$, are trivial. 
\end{rem2} 

\section{Applications to $C_2$-equivariant stable homotopy theory}\label{Sec:C2}

We now make some calculations in $C_2$-equivariant stable homotopy theory. Let $k \subseteq \r$ be a field with a real embedding. 

\begin{rem2}\label{Rem:namingconvention}
Throughout this section, we will freely use the names of elements in $\pi_{**}^{C_2}(H\f_2)$ and $\pi_{**}^{C_2}(H\z_2)$ used in \cite{GHIR19}. We briefly describe the relevant elements here.

The $C_2$-equivariant elements $\tau\in\pi_{0,-1}^{C_2}(H\f_2)$ and $\rho\in\pi_{-1,-1}^{C_2}(H\f_2)$ are the images under equivariant Betti realization of the $\r$-motivic elements with the same names.\footnote{There are also geometric models for $\tau$ and $\rho$, cf. \cite[Section 4]{May20}.}
The subalgebra $\f_2[\tau,\rho] \subseteq \pi_{**}^{C_2}(H\f_2)$ is sometimes called the \emph{positive cone}. 

There is another element, $\frac{\gamma}{\tau} \in \pi_{0,2}^{C_2}(H\f_2)$, which is infinitely $\rho$- and $\tau$-divisible but is not in the image of equivariant Betti realization. The infinitely divisible part of $\pi_{**}^{C_2}(H\f_2)$ is sometimes called the \emph{negative cone}. 

There are two sets of naming conventions for $C_2$-equivariant homotopy elements. We refer the reader to tables in Section \ref{subsection:conventions} for the correspondence.
\end{rem2}

We will use the following result of Heard to obtain $C_2$-equivariant slice differentials from known motivic differentials. We note that Heard actually proves a more general result, but we only use the following specialization in our work.

\begin{thm}[Heard]\label{Thm:Heard}\cite[Thms. 5.15-5.16]{Hea19}
Let $E \in \SH(k)$ be a motivic spectrum which is a localized quotient of $MGL$. Then $\re_{C_2} : (\SH(k), \Sigma^q_T \SH(k)^{eff}) \to (\SH^{C_2}, \Sigma^{2q}(\SH^{C_2})^{HHR})$ is compatible with the slice filtration at $E$.\footnote{We refer the reader to \cite{Hea19, HHR16} for details on the $C_2$-equivariant slice spectral sequences discussed in this section.}
\end{thm}

Heard applied this theorem to deduce effective slice spectral sequence differentials from the existence of HHR slice differentials. We will apply Heard's theorem in the opposite direction: starting with the $\r$-motivic effective slice spectral sequence differentials computed above, we will deduce the existence of differentials in the HHR slice spectral sequence. 

\begin{cor}\label{Cor:C2General}
Let $E \in \SH(k)$ be a motivic spectrum which is a localized quotient of $MGL$. Every $d_{2^{r-1}-1}$-differential in the ESSS for $E$ uniquely determines a  $d_{2^r-1}$-differentials in the HHR slice spectral sequence for $\re_{C_2}(E)$. 
\end{cor}

\begin{rem2}
	The difference in the lengths of the differentials in Corollary \ref{Cor:C2General} stems from a difference in the indexing conventions between the motivic slice spectral sequence and the equivariant slice spectral sequence. In particular, motivic $d_1$-differentials correspond to equivariant $d_3$-differentials. See the proof of \cite[Proposition 6.1]{Hea19} for more details. 
\end{rem2}

We will apply Corollary \ref{Cor:C2General} to the following $C_2$-spectra:

\begin{defin}
We define the following $C_2$-spectra using equivariant Betti realization.\footnote{Each spectrum can also be defined directly in $C_2$-spectra. See \cite{Ati66, Dug05} for $k\r$ and \cite{HHR16, HK01} for $k\r(n)$ and $BP\r\langle n \rangle$.}
\begin{enumerate}

\item Let
$$k\r := \re_{C_2}(kgl)$$
denote the \emph{$C_2$-equivariant slice cover of Atiyah's K-theory with reality}. 

\item Let
$$k\r(n) := \re_{C_2}(k(n))$$
denote the \emph{$n$-th connective Real Morava K-theory.}

\item Let
$$BP\r\langle n \rangle := \re_{C_2}(BPGL\langle n \rangle)$$
denote the \emph{$n$-th truncated Real Brown--Peterson spectrum.}

\end{enumerate}
\end{defin}

By \cite[Ex. 3.5]{Hea19}, the motivic spectra $kgl$, $k(n)$, and $BPGL\langle n \rangle$ are localized quotients of $MGL$. Combining Theorem \ref{Thm:Heard} and Corollary \ref{Cor:C2General}, we obtain:

\begin{cor}\label{Cor:AAA}
Let $E$ be $kgl$, $k(n)$, or $BPGL\langle n \rangle$. Every $d_{2^{r-1}-1}$-differential in the ESSS for $E$ uniquely determines a $d_{2^r-1}$-differential in the HHR slice spectral sequence for $\re_{C_2}(E)$. 
\end{cor}

We identified the differentials in the ESSS for $kgl$, $k(n)$, and $BPGL\langle n \rangle$ with $\rho$-Bockstein spectral sequence differentials in Theorems \ref{Thm:RCorrespondence} and \ref{Thm:BPGLCorrespondence}. Coupled with Corollary \ref{Cor:AAA}, this allows us to produce differentials in the HHR slice spectral sequences for $k\r$, $k\r(n)$, and $BP\r\langle n \rangle$. We spell out these consequences in the next few corollaries. 

\begin{cor}\label{Cor:kRdiff}
The nontrivial differentials in the HHR slice spectral sequence for $\pi_{**}^{C_2}k\r$ are determined via the Leibniz rule by
$$d_3(\tau^2) = \rho^3 \bar{v}_1 \quad \text{ and } \quad d_3 \left(\frac{\gamma}{\rho^3\tau^2}\right) = \frac{\gamma}{\tau^4}\bar{v}_1.$$
\end{cor}

\begin{proof}
The differential $d_1(\tau^2) = \rho^3 v_1$ in the ESSS for $kgl$ arising from Theorem \ref{Thm:RCorrespondence} gives rise to the differential $d_3(\tau^2) = \rho^3 \bar{v}_1$ in the HHR slice spectral sequence for $k\r$ using Corollary \ref{Cor:AAA}. 

We now prove the second differential. By the Leibniz rule,
	\begin{align*}
		0=d\left(\frac{\gamma}{\rho^3\tau^2}\cdot \tau^2\right)=&d\left(\frac{\gamma}{\rho^3\tau^2}\right)\cdot \tau^2+\frac{\gamma}{\rho^3\tau^2}\cdot d(\tau^2)
	\end{align*} 
	Therefore we have that 
	$$d\left(\frac{\gamma}{\rho^3\tau^2}\right)\cdot \tau^2=\frac{\gamma}{\rho^3\tau^2}\cdot d(\tau^2).$$
	Since the element $\tau^2$ support a $d_3$-differential, we deduce that the element $\frac{\gamma}{\rho^3\tau^2}$ lives to the $E_3$-page and supports a differential with target $\frac{\gamma}{\tau^4} \bar{v}_1$. This differential generates a family of $d_3$-differentials by multiplicity. 

By inspection, there is no room for further differentials. This completes the proof.
\end{proof}

\begin{cor}\label{Cor:kndiff}
There is a nontrivial differential in the HHR slice spectral sequence for $\pi_{**}^{C_2}k\r(n)$ of the form
$$d_{2^{n+1}-1}(\tau^{2^n}) = \rho^{2^{n+1}-1} \bar{v}_n.$$
\end{cor} 

\begin{proof}
By Theorem \ref{Thm:RCorrespondence} and Remark \ref{Rmk:RCorrespondence}, there is a 1-to-1 correspondence between $d_{2^{n+1}-1}$-differentials in the $\rho$-BSS and $d_{2^n-1}$-differentials in the ESSS for $k\r(n)$. By Proposition \ref{Prop:RhoBocksteinEQn}, the nontrivial differentials in the $\rho$-BSS are generated under $\rho$- and $v_n$-linearity by the differentials
$$d_{2^{n+1}-1}(\tau^{2^n}) = \rho^{2^{n+1}-1}v_n.$$
By Corollary \ref{Cor:AAA}, these $\r$-motivic differentials uniquely determine the stated $d_{2^{n+1}-1}$-differentials in the HHR slice spectral sequence for $k\r(n)$. 
\end{proof}

Similar arguments also apply to truncated Real Brown--Peterson spectra, where one uses Theorem \ref{Thm:BPGLCorrespondence} and \cite[Thm. 3.2]{Hil11} to identify the differentials in the ESSS for $BPGL\langle n \rangle$.  

\begin{cor}\label{Cor:BPRdiff}
There are nontrivial differentials in the HHR slice spectral sequence for $\pi_{**}^{C_2}BP\r\langle n \rangle$ of the form
$$d_{2^{i+1}-1}(\tau^{2^i}) = \rho^{2^{i+1}-1} \bar{v}_i$$
for $1 \leq i \leq n$. 
\end{cor}

\begin{rem2}
Using naturality and the quotient maps $BP\r \to BP\r\langle n \rangle$, the corollary recovers the $p$-local version of the $G=C_2$-case of the Hill--Hopkins--Ravenel Slice Differentials Theorem \cite[Thm. 9.9]{HHR16}.\footnote{They note this originally appeared in the work of Hu--Kriz \cite{HK01} and in unpublished work of Araki.}
\end{rem2}

\bibliographystyle{abbrv}
\bibliography{master}

\end{document}